\documentclass[12pt]{article}
\usepackage[margin=1in]{geometry}

\usepackage{tikz-cd}
\usepackage{tikz}
\usepackage{adjustbox}
\usepackage[algo2e,ruled,vlined]{algorithm2e}
\usepackage[normalem]{ulem}
\usepackage{multirow}
\usepackage{amsmath}
\usepackage{amssymb}
\usepackage{amsthm}
\usepackage{authblk}
\usepackage{graphicx}

\long\def\remove#1{}

\DeclareMathOperator{\inv}{\sf Inv}

\DeclareMathOperator{\cl}{\sf cl}
\DeclareMathOperator{\mo}{\sf mo}
\DeclareMathOperator{\pf}{\sf pf}

\DeclareMathOperator{\ds}{\sf ds}

\DeclareMathOperator{\im}{\sf im}

\DeclareMathOperator{\spint}{\overline{\cap}}

\newcommand{\cancel}[1]{}
\newcommand{\mv}[1]{\mathcal{V}_{#1}}
\newcommand{\mvint}[2]{\mv{#1} \spint \mv{#2}}

\newcommand{\md}[1]{ \mathcal{M}_{#1} }
\definecolor{colorblind-yellow}{RGB}{221,170,51}
\definecolor{colorblind-red}{RGB}{187,85,102}
\definecolor{colorblind-blue}{RGB}{0,68,136}
\definecolor{dark-gray}{RGB}{64,64,64}
\definecolor{medium-gray}{RGB}{114,114,114}
\definecolor{light-gray}{RGB}{190,190,190}

\newtheorem{theorem}{\sffamily Theorem}

\newtheorem{proposition}[theorem]{\sffamily Proposition}
\newtheorem{definition}[theorem]{\sffamily Definition}
\newtheorem{remark}[theorem]{\sffamily Remark}

\newif\ifpaper
\papertrue

\title{Persistence of Conley-Morse Graphs in Combinatorial Dynamical Systems} 

\begin{document}

\author[1]{Tamal K. Dey\thanks{tamaldey@purdue.edu} }

\author[2]{Marian Mrozek\thanks{marian.mrozek@uj.edu.pl}}

\author[1]{Ryan Slechta\thanks{rslechta@purdue.edu}}

\affil[1]{Department of Computer Science,
Purdue University, West Lafayette, USA}
\affil[2]{Division of Computational Mathematics, Faculty of Mathematics and Computer Science, Jagiellonian University, Krak\'{o}w, Poland}
\date{}

\maketitle

\setcounter{page}{1}
\begin{abstract}
Multivector fields provide an avenue for studying continuous dynamical systems in a combinatorial framework.  There are currently two approaches in the literature which use \emph{persistent homology} to capture changes in combinatorial dynamical systems. The first captures changes in the Conley index, while the second captures changes in the Morse decomposition. However, such approaches have limitations. The former approach only describes how the Conley index changes across a selected isolated invariant set though the dynamics can be much more complicated than the behavior of a single isolated invariant set. Likewise, considering a Morse decomposition omits much information about the individual Morse sets. In this paper, we propose a method to summarize changes in combinatorial dynamical systems by capturing changes in the so-called \emph{Conley-Morse graphs}. A Conley-Morse graph contains information about both the structure of a selected Morse decomposition and about the Conley index at each Morse set in the decomposition. Hence, our method summarizes the changing structure of a sequence of dynamical systems at a finer granularity than previous approaches.
\end{abstract}

\section{Introduction}
\label{sec:intro}

The theory of dynamical systems emerged from the need to predict the asymptotic behavior of solutions of differential equations. The field of topological dynamics (and the Conley theory \cite{Co78} in particular) has developed tools for analyzing the structure of \emph{invariant sets}, or the sets to which solutions limit. The Conley theory provides powerful tools for describing invariant sets, including the Conley index, Morse decompositions, Conley-Morse graphs, and connection matrices \cite{AKKMOP2009,BGHKMOP2012}. These tools have found wide applicability \cite{MIK19}, and they are particularly visible in the proof of the existence of chaos in the Lorenz system \cite{MiMr1995}. 

Despite its achievements, the theory of dynamical systems now requires new ideas to meet the needs of a scientific community that is increasingly dependent on data. In biology and the social sciences, the availability of huge amounts of data is in contrast with missing or poor classical mathematical models. Recently, combinatorial models have gained attention as a potential replacement for classical smooth models. The advantage of combinatorial models is that they facilitate direct algorithmic analysis, which makes them an ideal tool in the context of data. Forman's seminal work on combinatorial Morse theory \cite{Forman1998b,Forman1998a} introduced combinatorial models to dynamics via \emph{combinatorial vector fields}. Later, these were generalized to \emph{multivector fields} \cite{LKMW19,Mr2017}. In recent years, powerful constructions from topological dynamics, including Morse decompositions and the Conley index, have been adapted to this new combinatorial setting \cite{BKMW2020,KMW2016,LKMW19,Mr2017}.

In this combinatorial framework, dynamical systems can be analyzed via \emph{persistent homology}, or briefly, \emph{persistence}. Persistent homology is a popular data analysis tool that captures the changing homology of a sequence of spaces. Among the first attempts to use persistence in dynamics is the study of the changes of eigenspaces of self-maps \cite{EJM2015}. Recently, the authors in \cite{BEJM2020} used persistence to quantify recurrent behavior in a self-map. In \cite{DJKKLM19}, the authors use zigzag persistence \cite{zigzag} to capture the changing structure of \emph{Morse sets} in a Morse decomposition. Persistence is also used in \cite{DMS2020} to capture changes in the Conley index.

Unfortunately, the information given by only considering Morse sets or only considering the Conley index is incomplete. Given an isolating set, there may be multiple Morse sets, each of which is associated with a Conley index. This information is captured in a \emph{Conley-Morse graph}, which includes a vertex for each Morse set for a given Morse decomposition and the Conley index of the given Morse set. The Conley-Morse graph is a much more precise summary of the behavior of a combinatorial dynamical system than the Morse sets or Conley index alone. In this paper, we use persistence to capture the changing Conley-Morse graph of a sequence of multivector fields with a specified Morse decomposition. Our method tracks changes in combinatorial dynamical systems at a much finer level of detail than the methods in \cite{DJKKLM19,DMS2020}. Figure \ref{fig:changing-dynamical-system} is an example of a changing sequence of multivector fields, and Figure \ref{fig:conley-morse-graphs} shows the changing Conley-Morse graphs that correspond to these dynamical systems. Figure \ref{fig:old-approach} depicts how the approach in \cite{DMS2020} captures the changing behavior of combinatorial dynamical systems. In contrast, Figure \ref{fig:new-approach} depicts part of the approach in this paper, which summarizes the changing Conley index across a select set of sequences of Morse sets, shown in Figure \ref{fig:extracted-filtrations}. By inspecting the barcodes in Figure \ref{fig:old-approach} and \ref{fig:new-approach}, it is easy to see that this paper presents a much more detailed representation of the changing behavior of combinatorial dynamical systems than the approach in \cite{DMS2020}. We also briefly present a technique for obtaining noise-resilient index pairs in Section \ref{sec:multinode}. 

\begin{figure}[htbp]
\centering
\begin{tabular}{ccc}
  \includegraphics[height=34mm]{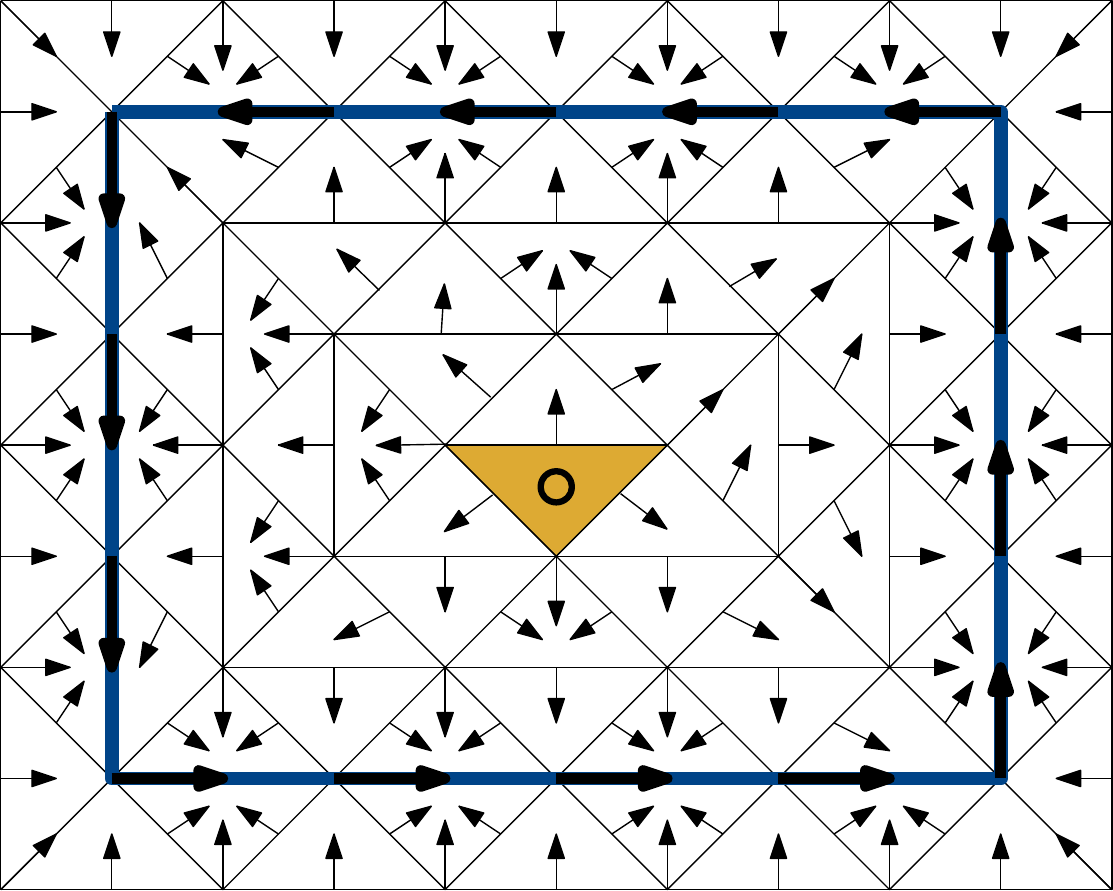}&
  \includegraphics[height=34mm]{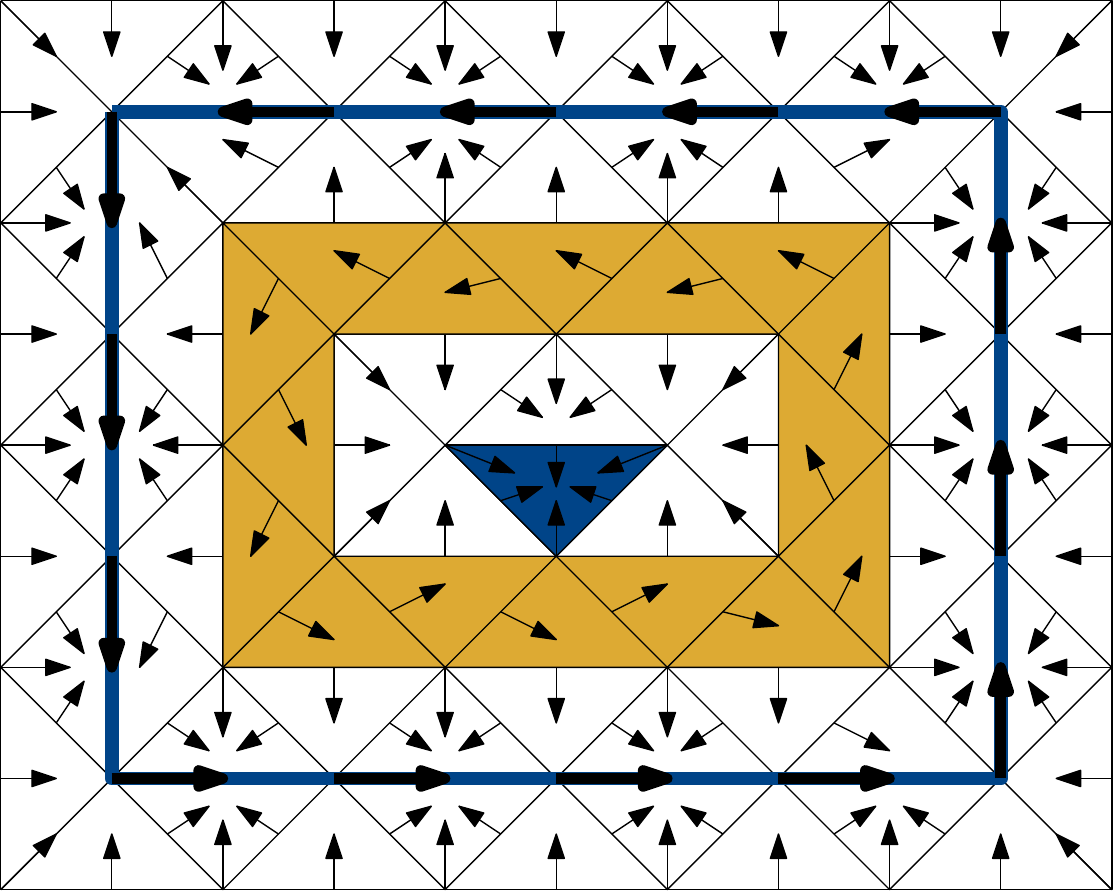}&
  \includegraphics[height=34mm]{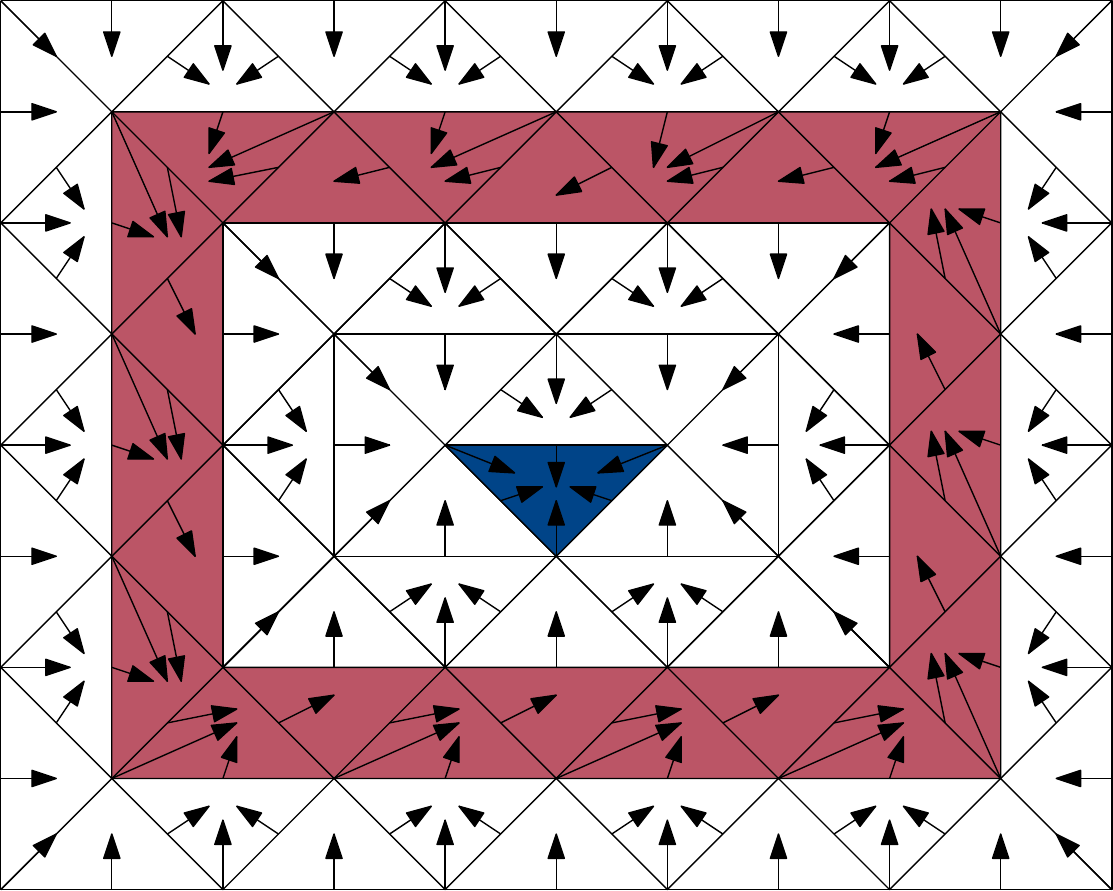}
\end{tabular}
\caption{Three multivector fields. On the left, a yellow repelling fixed point is surrounded by a blue periodic attractor. In the middle, the repelling fixed point has split into a yellow periodic repeller and a blue attracting fixed point. On the right, the periodic repeller has collided with the periodic attractor to form a red semistable limit cycle. In this example, all three multivector fields are significantly different from one another. In computing persistence, we will generally assume that there are several intermediate multivector fields, representing a gradual transition between the multivector fields shown here.}
\label{fig:changing-dynamical-system}
\end{figure} 

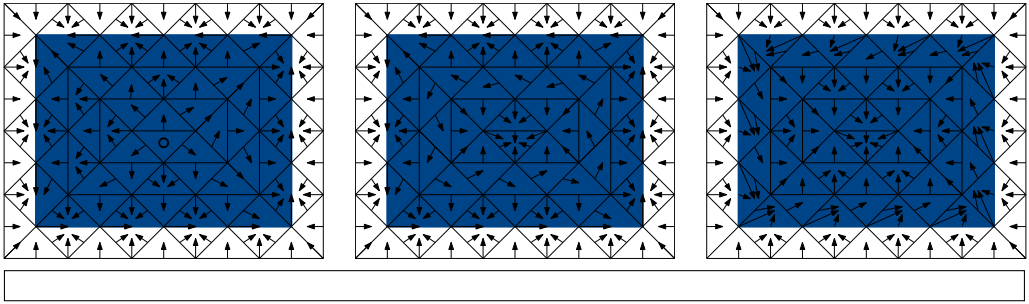
\begin{figure}[htbp]
\centering
\begin{tabular}{ccc}
  \includegraphics[height=34mm]{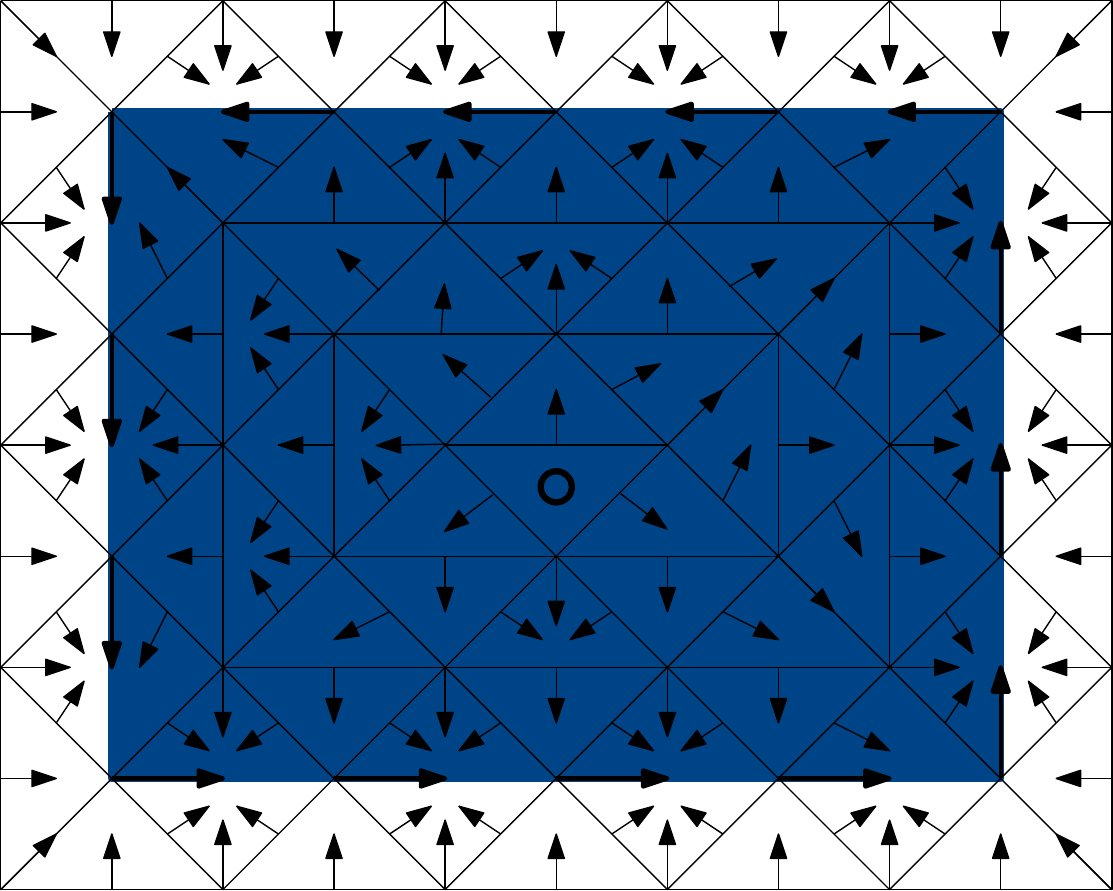}&
  \includegraphics[height=34mm]{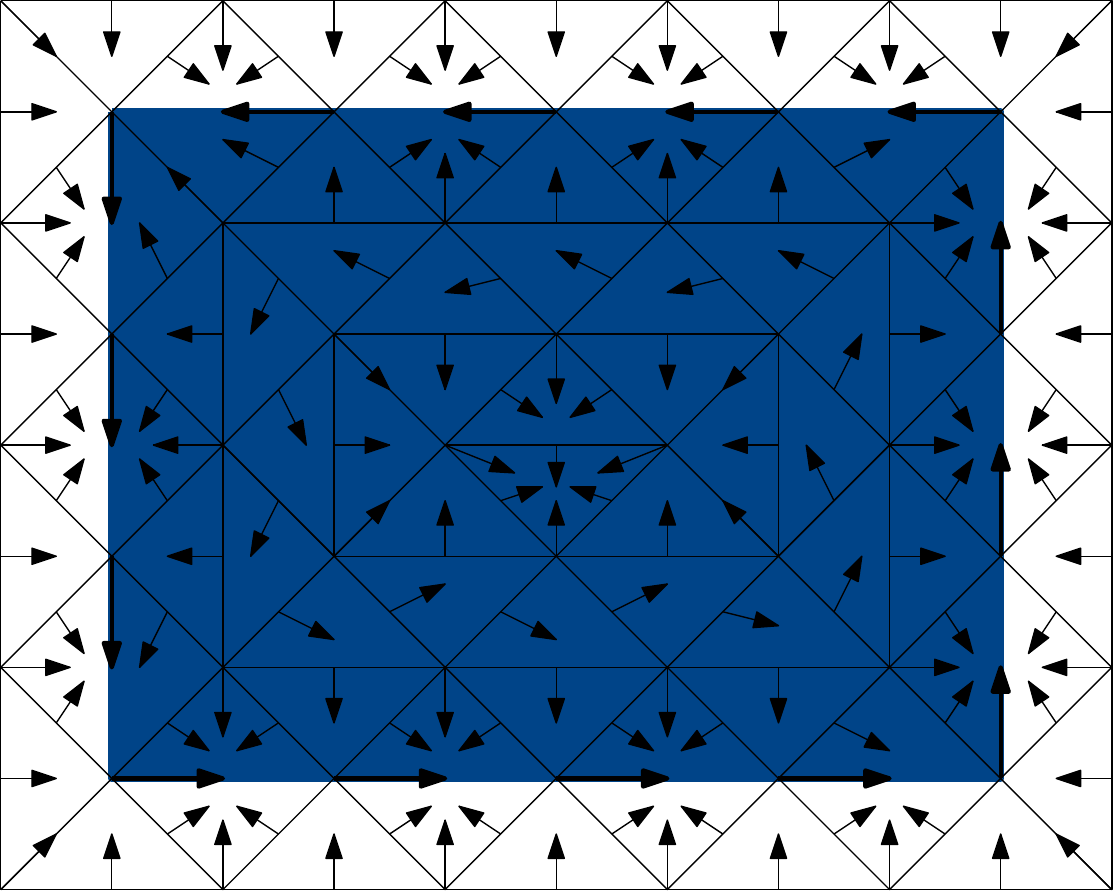}&
  \includegraphics[height=34mm]{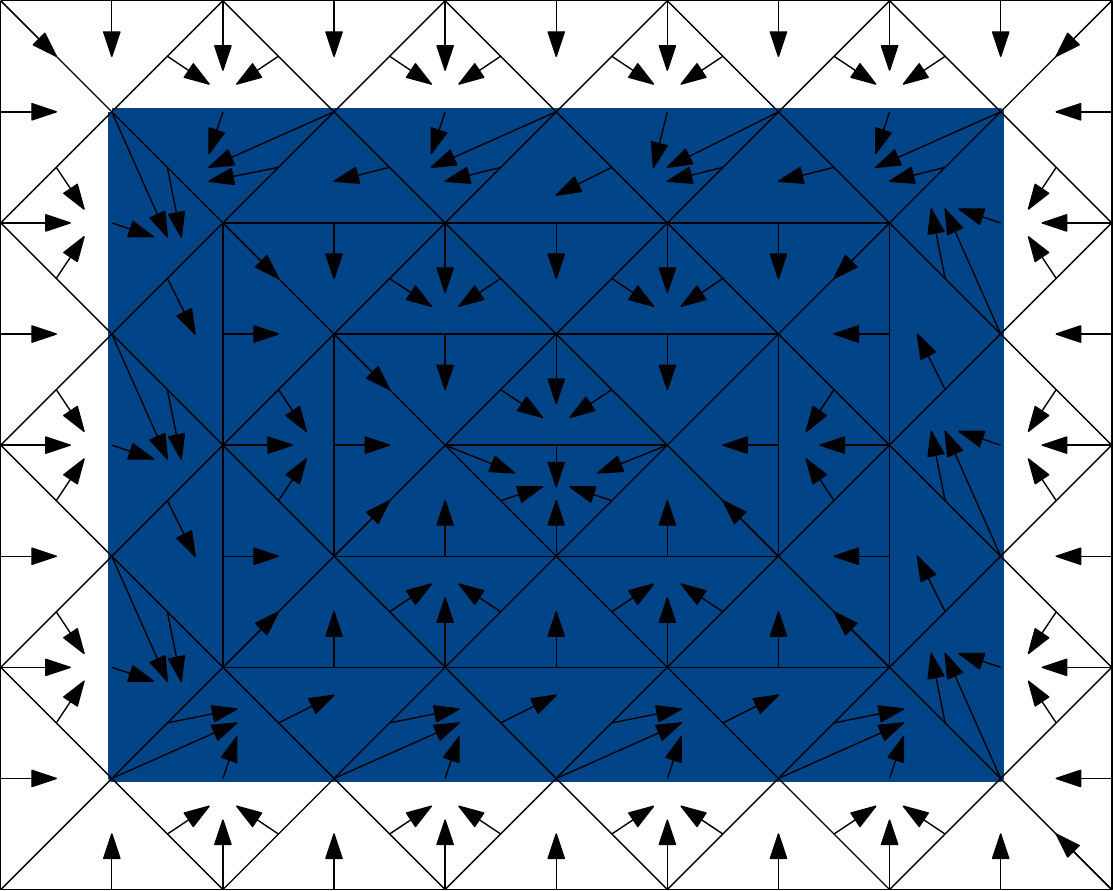}\\
  \multicolumn{3} {l} {
  \begin{tikzpicture}[outer sep = 0, inner sep = 0]
  \draw[fill=white,outer sep=0, inner sep=0] (0.0,-0.3) -- (13.56,-0.3) -- (13.56,0.1) -- (0.0,0.1) -- cycle;
  \end{tikzpicture} }
\end{tabular}
\caption{This figure illustrates the approach from \cite{DMS2020}, applied to capturing the changing structure of the multivector fields in Figure \ref{fig:changing-dynamical-system}. The approach in \cite{DMS2020} requires selecting a single isolated invariant set for each multivector field, with the canonical choice being the maximal isolated invariant set. In each multivector field, the maximal isolated invariant set is an attractor, highlighted in blue. Despite each multivector field giving rise to different dynamical systems, the maximal isolated invariant set is the same in each multivector field. Hence, computing the persistence using techniques from \cite{DMS2020} gives a single, $0$-dimensional bar, which we depict at the bottom in white.}
\label{fig:old-approach}
\end{figure} 

\begin{figure}[htpb]
    \centering
    \includegraphics[width=70mm]{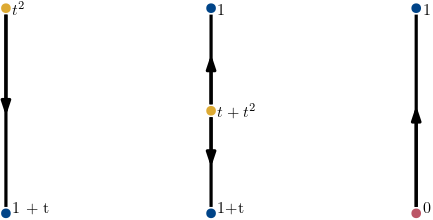}
    \caption{The Conley-Morse graph for the Morse decompositions (Definition~\ref{def:morse-decomp}) in Figure \ref{fig:changing-dynamical-system}, where the top vertices represent the fixed points. The colors of the vertices match the colors of the Morse sets with which they are drawn in Figure~\ref{fig:changing-dynamical-system}. Each label captures information about the Conley indices of the Morse sets.}
    \label{fig:conley-morse-graphs}
\end{figure}

\begin{figure}[htpb]
    \centering
    \includegraphics[width=92mm]{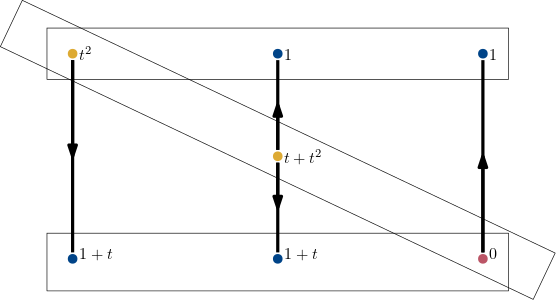}
    \caption{We extract a set of zigzag filtrations to capture the changing Conley indices in a sequence of Conley-Morse graphs. In this case, we extract three particular zigzag filtrations, which correspond to the sequences of Morse sets boxed in rectangles.}
    \label{fig:extracted-filtrations}
\end{figure}

\begin{figure}[htbp]
\centering
\begin{tabular}{ccc}
  \includegraphics[height=34mm]{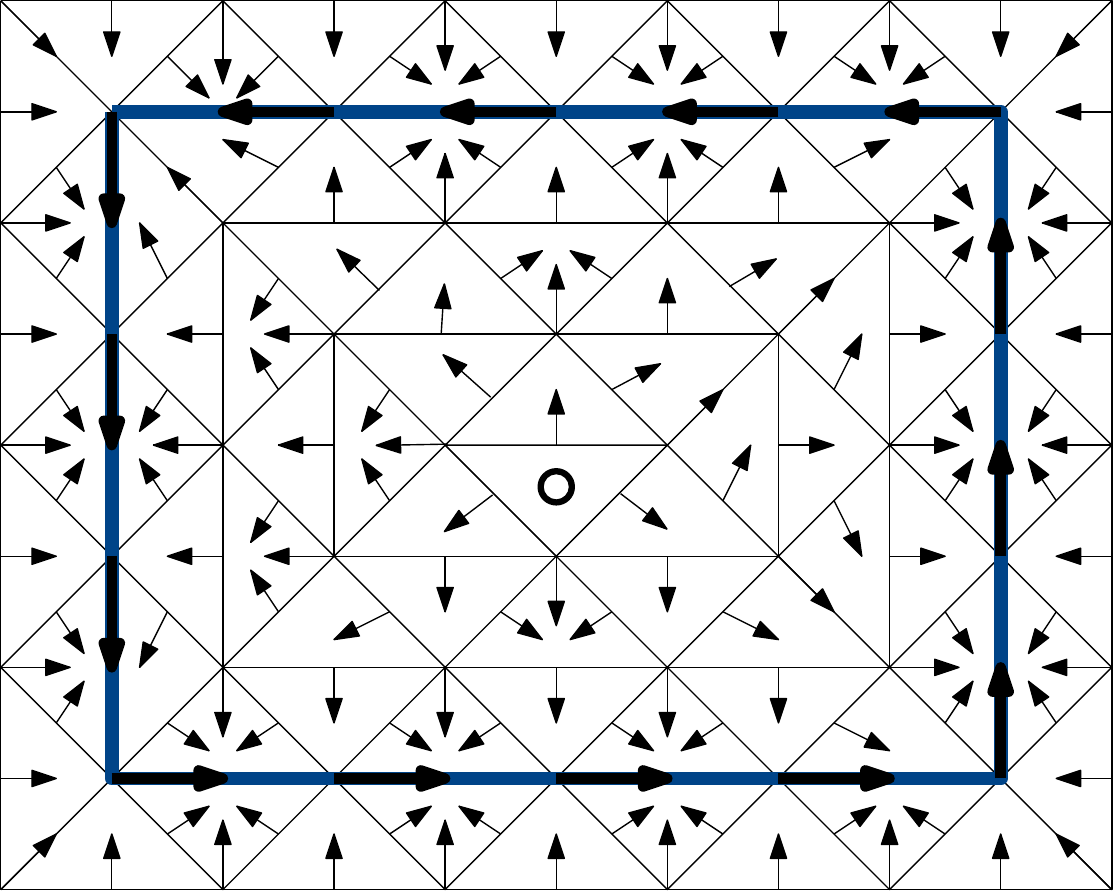}&
  \includegraphics[height=34mm]{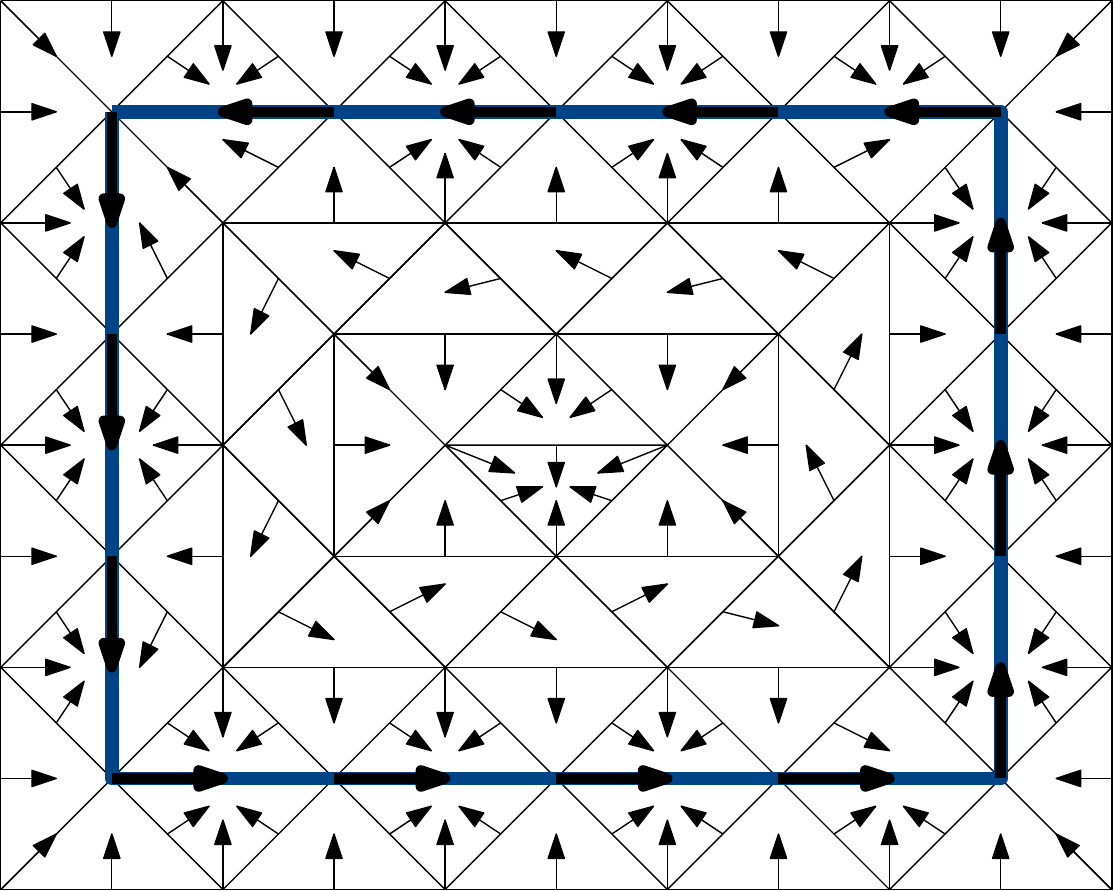}&
  \includegraphics[height=34mm]{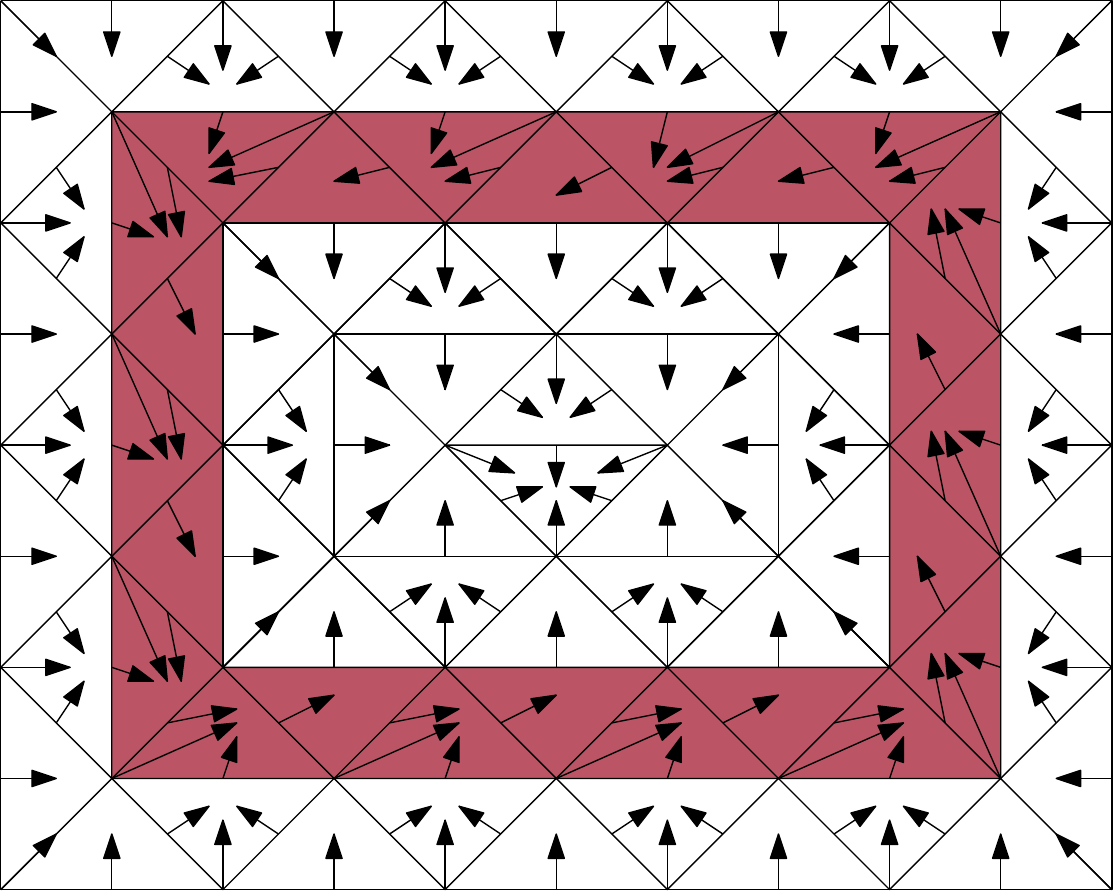}\\
  \multicolumn{3} {l} {
  \begin{tikzpicture}[outer sep = 0, inner sep = 0]
  \draw[fill=white,outer sep=0, inner sep=0] (0.0,-0.3) -- (9.34,-0.3) -- (9.34,0.1) -- (0.0,0.1) -- cycle;
  \node[align=left] at (1.3,-0.1) {Dimension: 0};
  \end{tikzpicture} }\\
  \multicolumn{3} {l} {
  \begin{tikzpicture}[outer sep = 0, inner sep = 0]
  \draw[fill=light-gray,outer sep=0, inner sep=0] (0.0,-0.3) -- (9.34,-0.3) -- (9.34,0.1) -- (0.0,0.1) -- cycle;
  \node[align=left] at (1.3,-0.1) {Dimension: 1};
  \end{tikzpicture} }\\
  \includegraphics[height=34mm]{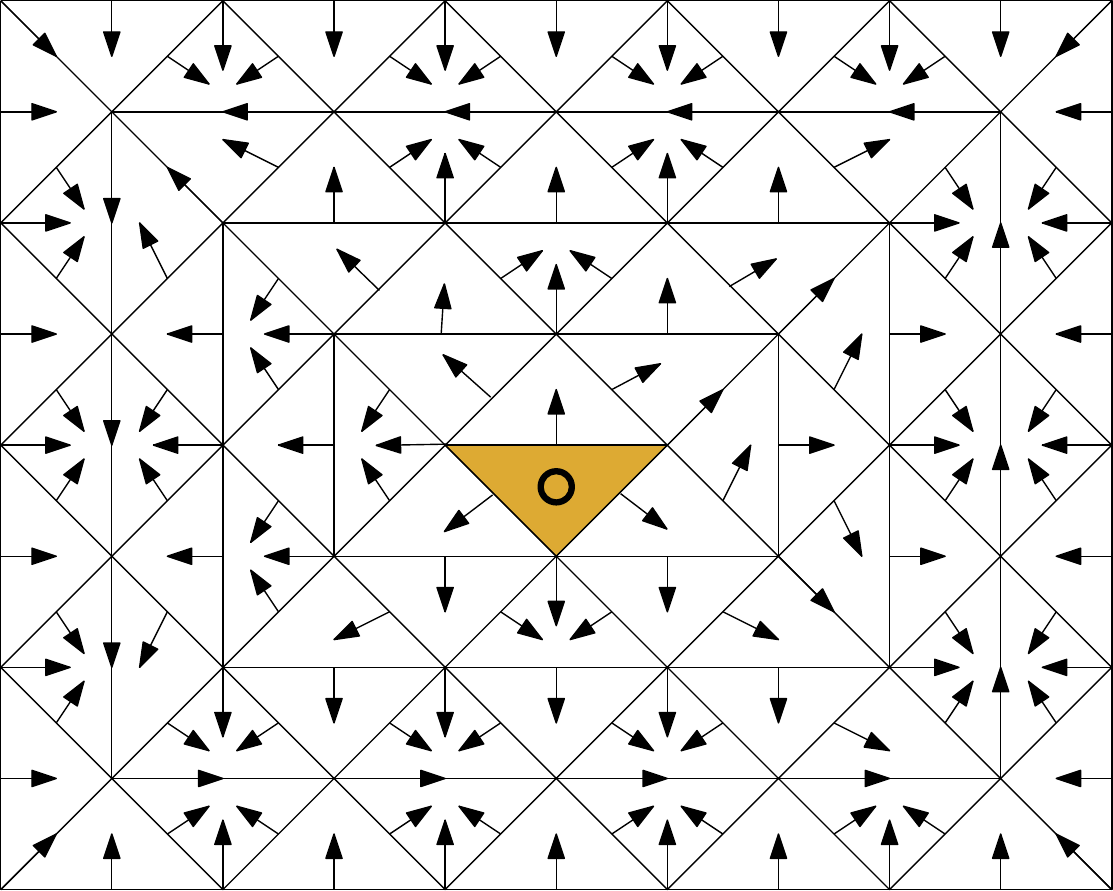}&
  \includegraphics[height=34mm]{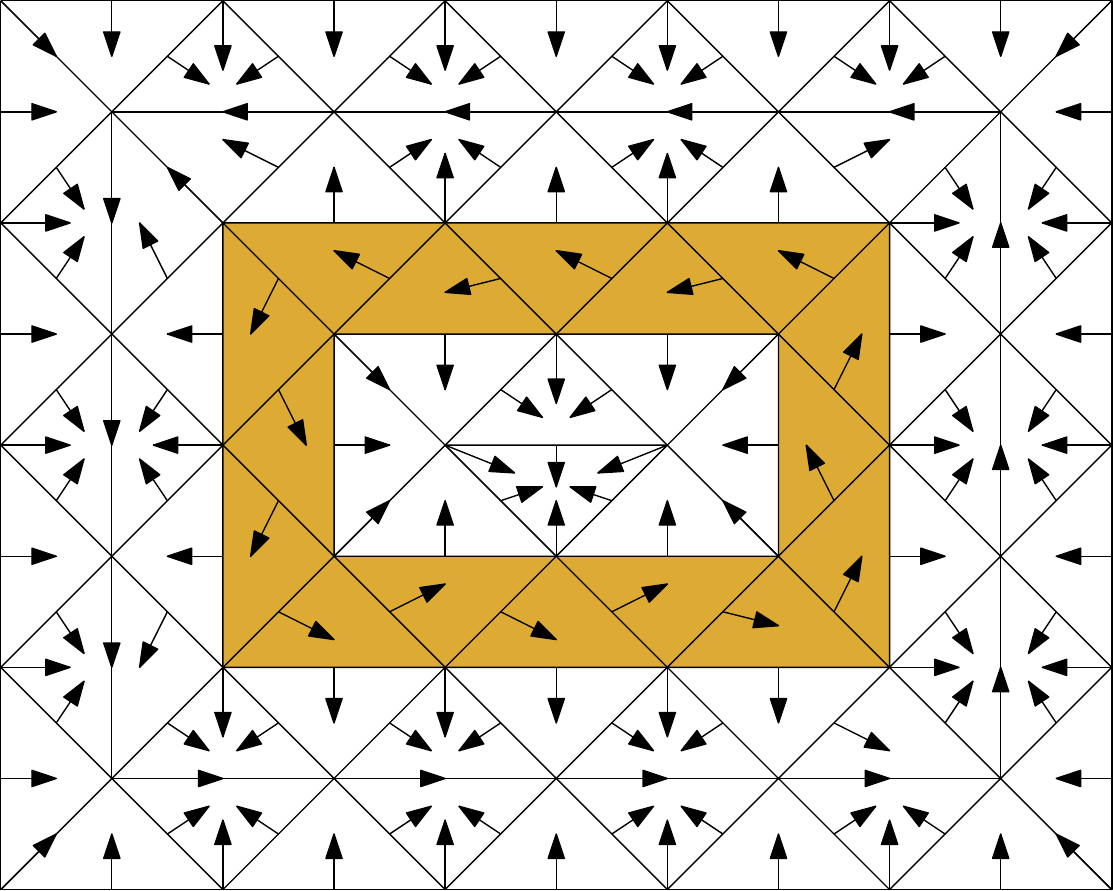}&
  \includegraphics[height=34mm]{fig/double-bifurcation-semistable-semistable.pdf}\\
  \multicolumn{3} {l} {
  \begin{tikzpicture}[outer sep = 0, inner sep = 0]
  \draw[fill=dark-gray,outer sep=0, inner sep=0] (0.0,-0.3) -- (9.34,-0.3) -- (9.34,0.1) -- (0.0,0.1) -- cycle;
  \node[align=left,color=white] at (1.3,-0.1) {Dimension: 2};
  \end{tikzpicture} }\\
  \multicolumn{3} {c} {
  \begin{tikzpicture}[outer sep = 0, inner sep = 0]
  \draw[fill=light-gray,outer sep=0, inner sep=0] (0.0,-0.3) -- (5.12,-0.3) -- (5.12,0.1) -- (0.0,0.1) -- cycle;
  \node[align=left] at (1.3,-0.1) {Dimension: 1};
  \end{tikzpicture} }\\
  \includegraphics[height=34mm]{fig/double-bifurcation-base-repeller.pdf}&
  \includegraphics[height=34mm]{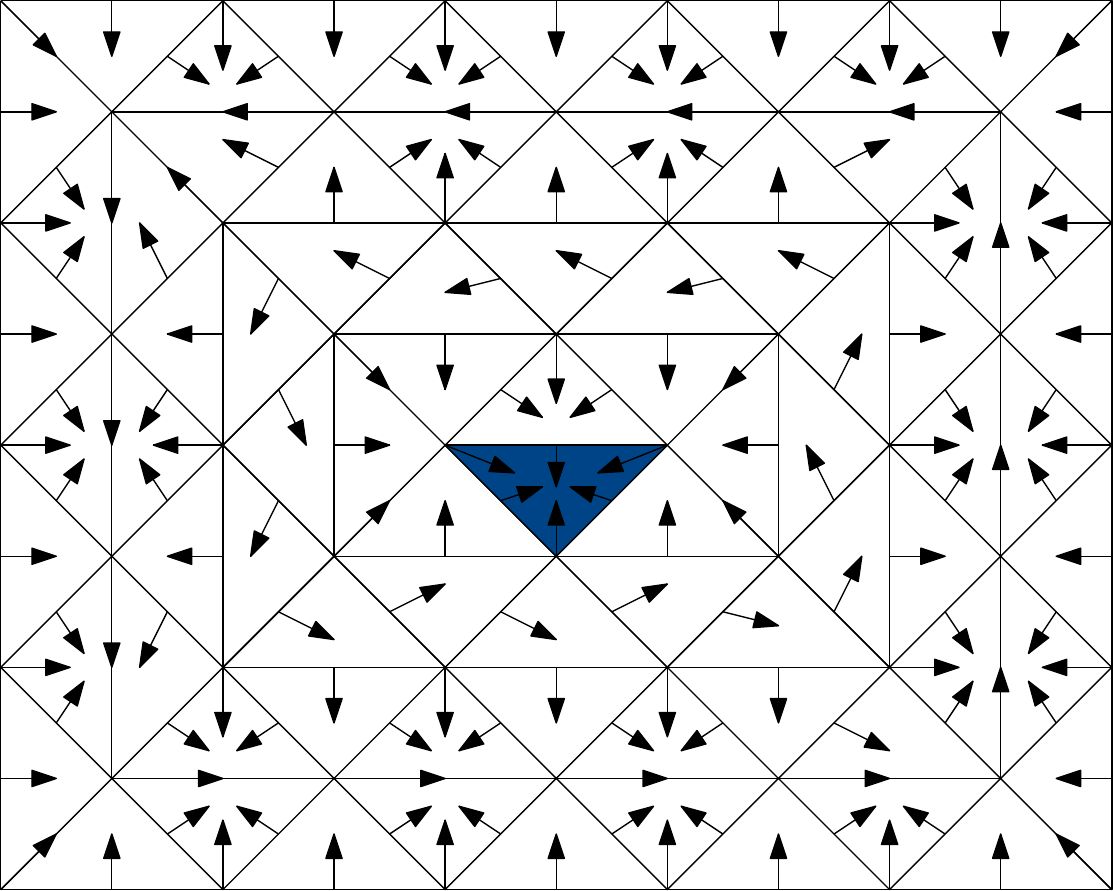}&
  \includegraphics[height=34mm]{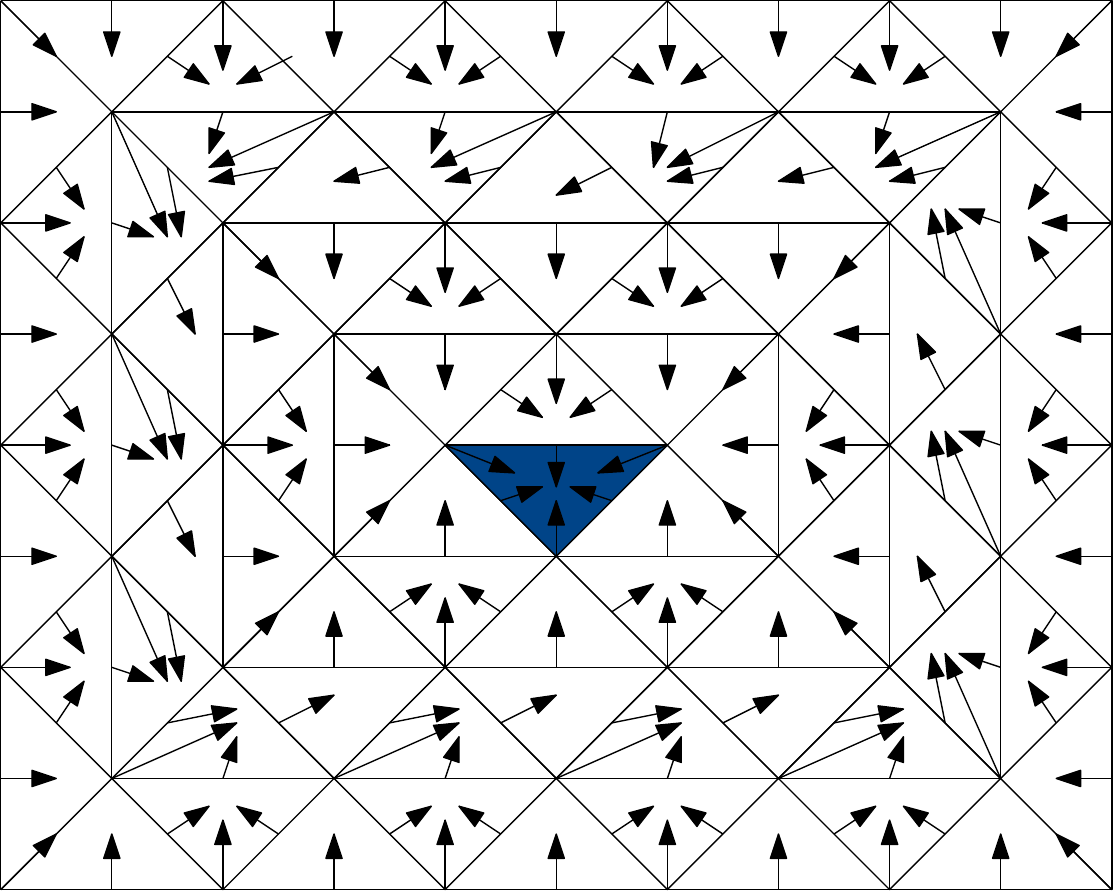}\\
  \multicolumn{3} {l} {
  \begin{tikzpicture}[outer sep = 0, inner sep = 0]
  \draw[fill=dark-gray,outer sep=0, inner sep=0] (0,-0.3) -- (4.22,-0.3) -- (4.22,0.1) -- (0,0.1) -- cycle;
  \node[align=left,color=white] at (1.3,-0.1) {Dimension: 2};
  \end{tikzpicture} }\\
  \multicolumn{3} {r} {
  \begin{tikzpicture}[outer sep = 0, inner sep = 0]
  \draw[fill=white,outer sep=0, inner sep=0] (4.22,-0.3) -- (13.56,-0.3) -- (13.56,0.1) -- (4.22,0.1) -- cycle;
  \node[align=left] at (5.5,-0.1) {Dimension: 0};
  \end{tikzpicture} }
\end{tabular}
\caption{Illustrating our approach for computing the changing structure of the multivector fields in Figure \ref{fig:changing-dynamical-system}. By extracting a specific set of zigzag filtrations (Section~\ref{sec:persist}), we can compute a set of barcodes which represent the changing Conley indices of the Morse sets (Section~\ref{sec:barcode}). Each row represents an extracted filtration, and the barcode for the filtration is depicted below it. White bars are $0$-dimensional, light gray bars are $1$-dimensional, and dark gray bars are $2$-dimensional. In addition, we extract barcodes representing the changing structure of the Conley-Morse graph (not pictured). We collate all of these bars into a single barcode by removing redundancy in Section \ref{sec:barcode}. }
\label{fig:new-approach}
\end{figure} 
\section{Preliminaries}
\label{sec:prelim}

In this section, we review some background on combinatorial dynamical systems. We assume that the reader is familiar with simplicial homology and persistent homology. Standard references for simplicial homology are \cite{hatcher,munkres}. For background on persistence, see \cite{EH2010}. 

\subsection{Multivector Fields}

Throughout this paper, we restrict our attention to simplicial complexes of arbitrary dimension. For a simplicial complex $K$, we use $\leq$ to denote the face relation, that is, $\sigma \leq \tau$ if $\sigma$ is a face of $\tau$. We define the \emph{closure} $\cl(\sigma)$ of $\sigma$ as $\cl( \sigma ):= \{ \tau \; | \tau \leq \sigma \}$ and extend the notion to a set of simplices $A \subseteq K$ as $\cl(A) := \cup_{\sigma \in A} \cl( \sigma )$. The set $A$ is \emph{closed} if $A = \cl(A)$. 

The face relation $\leq$ also induces a notion of convexity. A set $A \subseteq K$ is \emph{convex} if and only if for each $\sigma_1, \sigma_2 \in A$, there does not exist a $\sigma \in K \setminus \{\sigma_1, \sigma_2\}$ such that $\sigma_1 \leq \sigma \leq \sigma_2$. We say that a convex subset of $K$ is a \emph{multivector}. However, we are primarily interested in sets of multivectors called \emph{multivector fields} and relationships between them. 
\begin{definition}[Multivector Field, Refinement, Intersection Field]
A multivector field $\mv{}$ is a partition of a simplicial complex $K$ into multivectors. 
If $\mv{1}$ and $\mv{2}$ are multivector fields on $K$ such that for all $V \in \mv{1}$, there exists a $V' \in \mv{2}$ where $V \subseteq V'$, then $\mv{1}$ is a \emph{refinement} of $\mv{2}$. This relationship is denoted $\mv{1} \sqsubseteq \mv{2}$. Any two multivector fields $\mv{1}$, $\mv{2}$ over a simplicial complex $K$ give an \emph{intersection field} $\mvint{1}{2} := \{ V_1 \cap V_2 \; | \; V_1 \in \mv{1}, V_2 \in \mv{2} \}$.
\label{def:multivector}
\end{definition}
It is easy to check that $\mvint{1}{2}$ is always a multivector field, and $\mvint{1}{2}\sqsubseteq \mv{1}, \mv{2}$. We include an example of three multivector fields in Figure \ref{fig:changing-dynamical-system}. An intersection field is shown in Figure \ref{fig:spurious-pers}. A simplex $\sigma \in V$ is \emph{maximal in} $V$ if there does not exist a $\sigma' \in V$ where $\sigma \leq \sigma'$. We say that $\sigma$ is \emph{nonmaximal in} $V$ if there does not exist such a $\sigma' \in V$. In all figures, we draw a multivector $V$ by drawing an arrow from each nonmaximal $\tau$ in $V$ to each maximal $\sigma$ in $V$ where $\tau \leq \sigma$. The astute reader will notice that in general, multivectors need not be connected. All of the results in this paper hold for disconnected multivectors. However, in the interest of legibility, all of our examples will depict multivector fields with connected multivectors.  

In \cite{Mr2017}, a notion of dynamics on multivector fields is introduced. These dynamics takes the form of a multivalued map $F_{\mv{}} \; : \; K \multimap K$. Following~\cite{DJKKLM19}, we have $F_{\mv{}}(\sigma) = [ \sigma ]_{\mv{}} \cup \cl( \sigma )$, where $[ \sigma ]_{\mv{}}$ denotes the unique multivector in $\mv{}$ that contains $\sigma$. Each simplex $\sigma$ is a fixed point under $F_{\mv{}}$, that is, $\sigma \in F_{\mv{}}( \sigma )$. A finite sequence of simplices $\sigma_1, \sigma_2, \ldots, \sigma_n$ is a \emph{path} if for $i = 2, \ldots, n$, we have $\sigma_i \in F_{\mv{}}( \sigma_{i-1} )$. A bi-infinite sequence $\ldots, \sigma_{-1}, \sigma_0, \sigma_1, \ldots$ is a \emph{solution} if for all $i$, $\sigma_{i + 1} \in F_{ \mv{} }( \sigma_i )$. We often write paths as functions $\rho \; : \; \mathbb{Z} \cap [a,b] \to K$ and solutions as functions $\rho \; : \; \mathbb{Z} \to K$. As a solution $\rho$ exists relative to a multivector field $\mv{}$, we say that $\rho$ is a solution under $\mv{}$. An undesirable consequence of the definition of a solution is that every $\sigma \in K$ gives a solution $\rho$ where $\rho( \mathbb{Z} ) = \sigma$. This does not reflect intuition from differential equations. To correct for  this, the authors in \cite{LKMW19} introduced a notion of an \emph{essential solution}, which requires the notion of a \emph{critical} multivector. For convenience, we denote the \emph{mouth} of a multivector $V$ as $\mo(V) := \cl(V) \setminus V$. 
\begin{definition}[Critical and Regular Multivectors]
A multivector $V$ is \emph{critical} if there exists an integer $k \geq 0$ such that the relative homology group $H_k( \cl(V), \mo(V) )$ is nontrivial. If $V$ is not critical, then it is \emph{regular}.
\end{definition}
All homology groups in this paper are simplicial with coefficients taken from a finite field.
For a multivector $V$, both $\mo(V)$ and $\cl(V)$ are complexes, so $H_k( \cl(V), \mo(V))$ is well defined. 

Informally, an essential solution is allowed to stay in a critical multivector for infinite
time, but it must enter and exit each regular multivector that it visits.
\begin{definition}[Essential Solution]
A solution $\rho \; : \; \mathbb{Z} \to K$ under $\mv{}$ is an \emph{essential solution} if for each $i$ where $[ \rho( i )  ]_{\mv{}}$ is regular, there exists a pair $i^+, i^- \in \mathbb{Z}$, satisfying $i^- < i < i^+$, such that $[ \rho(i^-) ]_{ \mv{} } \neq [ \rho(i) ]_{ \mv{} }$ and $[ \rho(i^+) ]_{ \mv{} } \neq [ \rho(i) ]_{ \mv{} }$.
\end{definition}
We use essential solutions to define the invariant part of a set.
\begin{definition}[Invariant Part]
Let $A \subseteq K$ and let $\mv{}$ be a multivector field on $K$. The \emph{invariant part} of $A$, denoted $\inv_{\mv{}}(A)$, is the set of simplices $\sigma \in A$ such that there exists an essential solution $\rho \; : \; \mathbb{Z} \to K$ where $\rho(i) = \sigma$ for some $i\in \mathbb{Z}$ and $\rho( \mathbb{Z} ) \subseteq A$. 
\end{definition}
If the field $\mv{}$ is clear from context, we use the notation $\inv(A)$. A set $S \subseteq K$ is an \emph{invariant set} if $\inv_{\mv{}}( S ) = S$. In particular, $S$ is an invariant set \emph{under} $\mv{}$. Sometimes, we call $\inv(A)$ the \emph{maximal invariant set in} $A$. An invariant set $S$ is $\mv{}$-compatible if $S$ can be written as a union of multivectors in $\mv{}$. We are particularly interested in \emph{isolated} invariant sets.
\begin{definition}[Isolated Invariant Sets, Isolating Sets]
Let $S$ denote an invariant set under $\mv{}$, and let $N$ denote a closed set. If 
\begin{enumerate}
    \item $S$ is $\mv{}$-compatible, and 
    \item Each path $\rho \; : \; \mathbb{Z} \cap [a,b] \to N$ with $\rho(a), \rho(b) \in S$ has the property that $\im( \rho ) \subseteq S$,
\end{enumerate}
then $S$ is an \emph{isolated invariant set}. We also say that $S$ is \emph{isolated} by $N$, and $N$ is an \emph{isolating set}\footnote{It was termed `isolating neighborhood' in earlier papers such as~\cite{DJKKLM19,DMS2020,Mr2017}.} for $S$.
\end{definition}
We say that ``$S$ is an isolated invariant set in $N$'' to imply that $S$ is isolated by $N$. This paper heavily uses \emph{Morse decompositions} of isolated invariant sets. For an essential solution $\rho \; : \; \mathbb{Z} \to K$, we use the notation $\alpha(\rho) := \bigcap_{i=1}^\infty \rho(-\infty,-i]$ and $\omega(\rho) = \bigcap_{i=1}^\infty \rho[i,\infty)$. Because we take $K$ to be finite, $\alpha(\rho) \neq \emptyset$ and $\omega(\rho)\neq \emptyset$. Intuitively, $\alpha(\rho)$ and $\omega(\rho)$ capture the behavior at the beginning and end of an essential solution. 

\begin{definition}[Morse Decomposition]
Let $S$ denote an isolated invariant set in $N$ and $(\mathbb{P}, \leq)$ a finite poset. The collection $\md{} = \{M_p \; | \; p \in \mathbb{P}\}$ is called a \emph{Morse decomposition} of $S$ if both of the following conditions are satisfied: 
\begin{enumerate}
    \item $\md{}$ is a family of mutually disjoint, isolated invariant subsets of $S$, and
    \item For every essential solution $\rho \; : \; \mathbb{Z} \to S$ either $\im( \rho ) \subseteq M_r$ for an $r \in \mathbb{P}$ or there exist $p,q \in \mathbb{P}$ such that $q > p$, $\alpha(\rho) \subseteq M_q$, and $\omega( \rho ) \subseteq M_p$. 
\end{enumerate}

An element of a Morse decomposition is called a \emph{Morse set}. For Morse sets $M_p, M_q \in \md{}$, we often abuse notation and write $M_p \leq M_q$ if $p \leq q$. If the only Morse decomposition for an isolated invariant set $S$ is $\md{} = \{S\}$, then $S$ is \emph{minimal}.
\label{def:morse-decomp}
\end{definition}

In Figure \ref{fig:old-approach}, we have three isolated invariant sets. Figure \ref{fig:changing-dynamical-system} contains a Morse decomposition for each of the three. For example, one Morse decomposition for the blue rectangle in the center of Figure \ref{fig:old-approach} is given by the blue periodic attractor, the yellow periodic repeller, and the blue attracting fixed point that are depicted in the center of Figure \ref{fig:changing-dynamical-system}. 

In this paper, we will frequently use a particular type of Morse decomposition called a \emph{minimal} Morse decomposition. A Morse decomposition $\md{}$ is minimal if each $M \in \md{}$ is minimal. Fortunately, there is a simple characterization of minimal isolated invariant sets.

\begin{proposition}{\cite[Proposition 6.7]{LKMW19}}
Let $\mv{}$ denote a multivector field over $K$, and let $S$ denote an isolated invariant set under $\mv{}$. The set $S$ is minimal if and only if for all $\sigma,\tau \in S$, there exists a path $\rho \; : \; [0, n] \to S$ where $\rho( 0 ) = \sigma$ and $\rho( n ) = \tau$. 
\label{prop:min}
\end{proposition}

\subsection{Conley Index}
The Conley index has been defined in different contexts~\cite{Co78,LKMW19,Mr2017}. Here we use the definition in \cite{LKMW19}, where it is defined using essential solutions. First, we define the \emph{index pair}. 
\begin{definition}[Index Pair]
Let $S$ be an isolated invariant set under $\mv{}$, and let $P$ and $E$ be closed sets with $E \subseteq P$. If all of the following hold, then $(P,E)$ is an \emph{index pair for} $S$:
\begin{enumerate}
\item $F_{\mv{}}( E ) \cap P \subseteq E$
\item $F_{\mv{}} ( P \setminus E ) \subseteq P$
\item $S = \inv_{\mv{}}( P \setminus E )$
\end{enumerate}
\label{def:ip}
\end{definition}
We can use index pairs to define the Conley Index. 
\begin{definition}[Conley Index]
Let $(P,E)$ denote an index pair for an isolated invariant set $S$. The \emph{$k$-dimensional Conley Index} is given by the relative homology group $H_k(P, E)$. 
\end{definition}
\begin{figure}[htbp]
\centering
\begin{tabular}{cc}
  \includegraphics[height=50mm]{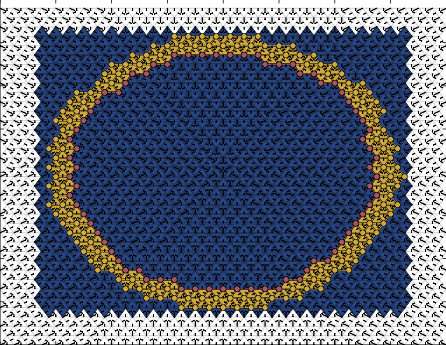}&
  \includegraphics[height=50mm]{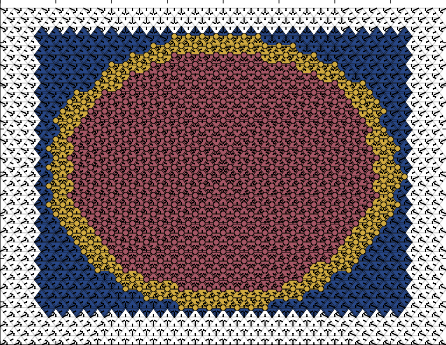}
\end{tabular}
\caption{ In both figures, we depict an isolated invariant set $S$ in yellow, and an isolating set $N$ for $S$ given by all colored simplices. On the left, we depict an index pair for $S$, where $P$ is given by the red and yellow simplices, and $E$ is given by the red simplices. On the right, we have an index pair for $S$ in $N$, where $P$ is given by the yellow and red simplices, and $E$ is given by the red simplices. The index pair for $S$ in $N$ is also an index pair in the sense of Definition \ref{def:ip}. For all integers $k \geq 0$, $H_k(P,E) = 0$. }
\label{fig:conley-index-examples}
\end{figure}
We depict two examples of index pairs in Figure \ref{fig:conley-index-examples}. In \cite{LKMW19}, the authors observed the following. 
\begin{proposition}{\cite[Proposition 5.16]{LKMW19}}
Let $(P,E)$ and $(P',E')$ denote index pairs for an isolated invariant set $S$. Then $H_k(P,E) = H_k(P',E')$. 
\end{proposition}
\begin{proposition}{\cite[Proposition 5.3]{LKMW19} }
The pair $(\cl(S), \mo(S) )$ is an index pair for the isolated invariant set $S$. 
\label{prop:ip}
\end{proposition}

In \cite{DMS2020}, the authors developed a method to summarize the changing Conley index of a series of invariant sets by using zigzag persistence. Given index pairs $(P_1, E_1)$, $(P_2, E_2)$ for isolated invariant sets under $\mv{1}$, $\mv{2}$, a natural approach to compute the changing Conley index is to use the relative zigzag filtration in Equation \ref{eqn:zigzagreg}. 
\begin{equation}
    (P_1, E_1) \supseteq (P_1 \cap P_2, E_1 \cap E_2) \subseteq (P_2, E_2).
    \label{eqn:zigzagreg}
\end{equation}
Unfortunately, $(P_1 \cap P_2, E_1 \cap E_2)$ need not be an index pair under $\mv{1}$, $\mv{2}$, or the intersection multivector field, $\mvint{1}{2}$ (see Definition~\ref{def:multivector}). Hence, if we extract the barcodes from such a relative zigzag filtration, we may not actually compute a changing Conley index. To rectify this, the authors in \cite{DMS2020} introduced a particular type of index pair.
\begin{definition}[Index Pair in $N$]
Let $S$ be an isolated invariant set and let $N$ denote an isolating set for $S$. The pair of closed sets $E \subseteq P \subseteq N$ is an index pair in $N$ for $S$ if all of the following conditions are met:
\begin{enumerate}
    \item $F_{\mv{}}( E ) \cap N \subseteq E$
    \item $F_{\mv{}}( P ) \cap N \subseteq P$
    \item $F_{\mv{}}( P \setminus E ) \subseteq N$
    \item $S = \inv( P \setminus E)$
\end{enumerate}
\end{definition}
The right image in Figure \ref{fig:conley-index-examples} is an index pair in $N$, where $N$ is given by the colored simplices, $P$ is given by the red and yellow simplices, and $E$ is given by the red simplices. The set $P \setminus E$ is an isolated invariant set, and $P\setminus E$ is equal to the set of yellow simplices.
\begin{proposition}{\cite[Theorem 8]{DMS2020}}
Let $(P,E)$ denote an index pair in $N$ for the isolated invariant set $S$. Then $(P,E)$ is an index pair for $S$ in the sense of Definition \ref{def:ip}.
\end{proposition}
Furthermore, index pairs in $N$ have the property that their intersections give index pairs.
\begin{theorem}{\cite[Theorem 10]{DMS2020} }
Let $(P_1,E_1)$ denote an index pair in $N$ for $S_1$ under $\mv{1}$ and let $(P_2,E_2)$ denote an index pair in $N$ for $S_2$ under $\mv{2}$. The pair $(P_1 \cap P_2, E_1 \cap E_2)$ is an index pair in $N$ for $\inv( (P_1 \cap P_2 ) \setminus (E_1 \cap E_2) )$ under $\mvint{1}{2}$. 
\label{thm:interindexpair}
\end{theorem}
Hence, given a sequence of index pairs in $N$, one can consider the relative zigzag filtration:
\begin{equation}
    (P_1, E_1) \supseteq (P_1 \cap P_2, E_1 \cap E_2) \subseteq (P_2, E_2).
    \label{eqn:zigzaginn}
\end{equation}
Unlike the construction in Equation \ref{eqn:zigzagreg}, each pair in Equation \ref{eqn:zigzaginn} is an index pair. Hence, by using this approach, we get a barcode that actually corresponds to a changing Conley Index. 

Given an isolated invariant $S$ in an isolating set $N$, one can easily compute an index pair in $N$ for $S$ by using the $\emph{push forward}$. For a set $A \subset K$, $\pf_N(A)$ is the set of simplices in a closed set $N$ which can be reached from paths in $N$ and originating at a simplex in $A$. Formally, we have the following.
\begin{proposition}{\cite[Proposition 15]{DMS2020}}
If $S$ is an isolated invariant set in $N$, then the pair $(\pf_N(\cl(S)), \pf_N(\mo(S)))$ is an index pair in $N$ for $S$. 
\label{prop:pfip}
\end{proposition}

\subsection{Conley-Morse Graph}

In the previous subsections, we have presented the Morse decomposition and the Conley index. The Morse decomposition and the Conley index provide different information about an isolated invariant set. These two descriptors are often combined into the \emph{Conley-Morse graph}, which represents a Morse decomposition and contains information about the Conley indices of the Morse sets in the decomposition.
\begin{definition}[Conley-Morse graph]
Let $\md{}$ denote a Morse decomposition, and let $G$ denote the directed graph such that there is a bijection $f \; : \; \md{} \to V(G)$, and there exists a connection from $M \in \md{}$ to $M' \in \md{}$ if and only if there there exists a directed edge from $f(M)$ to $f(M')$. The \emph{Conley-Morse graph} for $\md{}$ is the graph $G$ where each vertex $f(M) = v \in V(G)$ is annotated with the Poincar\'{e} polynomial $\sum_{i=0}^m\beta_it^i$ where $m$ is the largest integer for which the $m$-dimensional Conley index is nontrivial and $\beta_i$ is the rank of the $i$-dimensional Conley index of $M$.
\end{definition}
\begin{figure}[htbp]
\centering
\begin{tabular}{cc}
  \includegraphics[width=55mm]{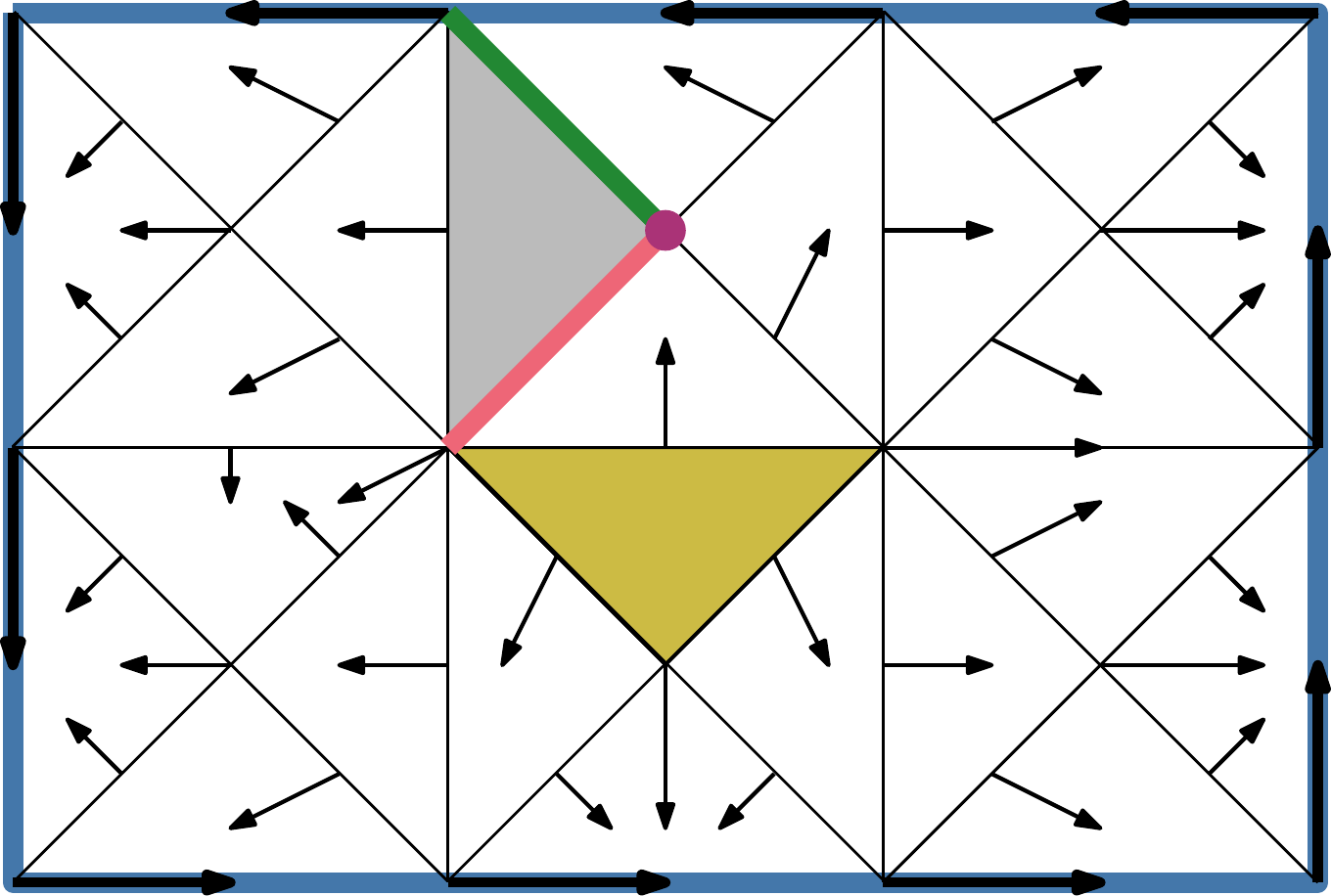}&
  \includegraphics[width=30mm]{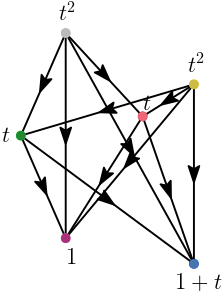}
\end{tabular}
\caption{ On the left, we show a multivector field and the minimal Morse decomposition for the maximal invariant set in $N$, where $N$ is the entire rectangle. The maximal invariant set in $N$ is also the entire rectangle because every colored simplex is critical and every white or black simplex is on a path from the golden critical triangle to the blue periodic attractor or from the gray critical triangle to the blue periodic attractor. Each of the six Morse sets in the minimal Morse decomposition is represented by a vertex in the Conley-Morse graph (right) with a matching color. Each vertex is annotated with a Poincar\'e polynomial that summarizes the Conley indices of the Morse set. }
\label{fig:ex-conley-morse}
\end{figure}
For convenience, we use $M$ to refer both to a Morse set and to its corresponding vertex in the Conley-Morse graph. We include an example of a Morse decomposition and the associated Conley-Morse graph in Figure \ref{fig:ex-conley-morse}.

\section{Conley-Morse Filtrations}
\label{sec:persist}

We now move to develop a method to compute a set of filtrations for a sequence of Conley-Morse graphs. These filtrations need to capture key features of the sequence: how the structure of the Conley-Morse graphs changes  throughout the sequence, and how the Conley index changes at individual vertices; see
Figures \ref{fig:conley-morse-graphs}, \ref{fig:extracted-filtrations}, and \ref{fig:new-approach}. Formally, we assume that we are given a sequence of Conley-Morse graphs $\{G_i\}_{i=1}^n$. Each Conley-Morse graph corresponds to a Morse decomposition $\{\md{i}\}_{i=1}^n$ of an isolated invariant set $\{S_i\}_{i=1}^n$ under a multivector field $\{\mv{i}\}_{i = 1}^n$. We assume that each $S_i$ is in some fixed isolating set $N$. We also assume that for each $M_j \in \md{i}$, we have an associated index pair $(P_j, E_j)$. This may seem particularly daunting, in that our approach requires a specified isolated invariant set and Morse decomposition for every multivector field. If a practitioner does not have a particular isolated invariant set or Morse decomposition in mind, a canonical choice is always to take the maximal invariant set $S_i$ in $N$ under $\mv{i}$ and to take the minimal Morse decomposition of $S_i$. The minimal Morse decomposition can be easily computed by converting a multivector field into a directed graph and computing maximal strongly connected components. For more details, see \cite{LKMW19}. In addition, we require that each Morse set $M$ is associated with a particular index pair $(P,E)$ in $N$. A canonical choice is to use Proposition \ref{prop:pfip} and let $P := \pf(\cl(M))$ and $E:= \pf(\mo(M))$, but this is in general not resilient to noise. We include a technique for computing noise-resilient index pairs in Section \ref{sec:multinode}.

In this section, we extract Conley index zigzag filtrations from a sequence of Conley-Morse graphs $G_1$, $G_2$, \ldots, $G_m$, corresponding to Morse decompositions $\md{1}$, $\md{2}$, \ldots, $\md{m}$. That is, we extract filtrations which represent the changing Conley index at the vertices of the Conley-Morse graph. We call these \emph{Conley-Morse} filtrations, and we postpone discussion of \emph{graph filtrations}, which capture the changing structure of the Conley-Morse graph itself, to the next section. A first approach to extracting Conley-Morse filtrations is to consider all possible zigzag filtrations via vertex sequences (three such sequences are shown in Figure~\ref{fig:extracted-filtrations}). Each vertex corresponds to an index pair, so for each sequence of vertices taken from consecutive Conley-Morse graphs, we can use Theorem \ref{thm:interindexpair} to get a relative zigzag filtration. If $|G_i| = n$ for all $i$, then there are $\Theta(n^m)$ possible zigzag filtrations. This is clearly intractable for mid-to-large values of $m$ and $n$. We aim to reduce the number of filtrations while still capturing the changing Conley index. Our key observation is that if $(P_1, E_1)$ and $(P_2,E_2)$ are index pairs, then clearly
a homology class persists from $(P_1,E_1)$ to $(P_2,E_2)$ only if $(P_1 \setminus E_1) \cap (P_2 \setminus E_2) \neq \emptyset$, because an empty intersection leaves no room for the class to be
present in both spaces under inclusions. 

Consider a sequence of index pairs in $N$, denoted $\{(P_i,E_i)\}_{i = a}^b$, where each $(P_i, E_i)$ is an index pair for a Morse set $M_i \in \md{i}$. Such a sequence is \emph{feasible} if for all $a \leq i < b$, $(P_i \setminus E_i) \cap (P_{i+1} \setminus E_{i+1}) \neq \emptyset$. A feasible sequence of index pairs is a \emph{maximal} sequence if there does not exist a $(P_{a-1},E_{a-1})$ such that $(P_{a-1} \setminus E_{a-1}) \cap (P_a \setminus E_a) \neq \emptyset$, and there does not exist a $(P_{b+1}, E_{b+1})$ such that $(P_{b} \setminus E_{b}) \cap (P_{b+1} \setminus E_{b+1}) \neq \emptyset$. Hence, each feasible sequence is contained in a maximal sequence. Each maximal sequence gives a relative zigzag filtration by intersecting consecutive index pairs (recall Theorem \ref{thm:interindexpair}). We call a relative zigzag filtration that is given by a maximal sequence a \emph{Conley-Morse filtration}, and we use the set of Conley-Morse filtrations to compute the changing Conley indices at the vertices of a sequence of Conley-Morse graphs. 

\begin{algorithm2e}
\SetAlgoLined
\KwIn{ Sequence of Conley-Morse graphs $G_i$ corresponding to Morse decompositions $\mathcal{M}_i$ for isolated invariant sets in $N$. Each $M_i \in \md{i}$ is associated with a unique index pair $(P,E)$.} 
\KwOut{ Set of all Conley-Morse filtrations for $\{G_i\}_{i=1}^n$}
    
    $alive\_seqs \gets \texttt{new set}()$
    
    $all\_seqs \gets \texttt{new set}()$
    
    \cancel{\For{ $M \in \md{1}$ }{
        $(P,E) \gets \texttt{FindIndexPair}( M )$
        
        $seq \gets \texttt{new Sequence()}$
        
        $\texttt{append}( seq, (P,E) )$
        
        $\texttt{add}(alive\_seqs,seq)$
    }}

    \For{$i \in \{1, \ldots, n\}$}
    {
        $to\_remove \gets \texttt{new set}()$
        
        $still\_alive \gets \texttt{new set}()$
        
        \For{$seq \in alive\_seqs$}
        {
            
            $(P',E') \gets \texttt{LastIndexPair}(seq)$
        
            \For{$M \in \md{i}$}
            {
            
                $(P,E) \gets \texttt{IndexPair}( M )$
                
                \If{ $(P \setminus E) \cap (P' \setminus E') \neq \emptyset$ }{
                    
                    $new\_seq \gets \texttt{copy}(seq)$
                    
                    $\texttt{append}(new\_seq, (P,E) )$
                    
                    $\texttt{add}( still\_alive, new\_seq )$
                    
                    $\texttt{IsInSequence}(M) \gets \texttt{True}$
                    
                    $\texttt{add}(to\_remove, seq )$
                }
            }
        }
        
        $dead\_seqs \gets alive\_seqs \setminus to\_remove$
        
        $alive\_seqs \gets still\_alive$
        
        $all\_seqs \gets all\_seqs \cup dead\_seqs$
        
        \For{$M \in \md{i}$}
        {
            \If{ $\texttt{IsInSequence}(M) = \texttt{False}$}
            {
                $seq \gets \texttt{new Sequence}()$
                
                $\texttt{append}(seq, \texttt{IndexPair}(M) )$
                
                $\texttt{add}(alive\_seqs,seq)$
            }
        }
    }
    
    $all\_seqs \gets all\_seqs \cup alive\_seqs$
    
    $filtrations \gets \texttt{ConvertToFiltrations}( all\_seqs )$
    
    \Return{ $filtrations$ }

 \caption{$\texttt{FindConleyMorseFiltrations}( \{G_i\}_{i = 1}^n, \{\md{i}\}_{i = 1}^n)$ }
 \label{alg:find-max-zigzags}
\end{algorithm2e}

Algorithm \ref{alg:find-max-zigzags} is our formal approach for computing the set of Conley-Morse filtrations. For completeness, we prove that it successfully computes the set. 

\begin{proposition}
Given a sequence of Conley-Morse graphs $\{G_i\}_{i=1}^n$ corresponding to a sequence of Morse decompositions $\{ \md{i} \}_{i=1}^n$ of isolated invariant sets in $N$, where each $M_i \in \md{i}$ is associated with a unique index pair $(P,E)$, Algorithm \ref{alg:find-max-zigzags} outputs all Conley-Morse filtrations. 
\end{proposition}
\begin{proof}
Note that Algorithm \ref{alg:find-max-zigzags} attempts to find all Conley-Morse filtrations by first finding all maximal sequences. Hence, it is sufficient to show that it correctly finds all maximal sequences, because it is trivial to convert these into Conley-Morse filtrations. Note that Algorithm \ref{alg:find-max-zigzags} constructs sequences by incrementally adding index pairs to already-started sequences. By inspection, we note that the algorithm only constructs a new sequence with the index pair $(P,E)$ from a Morse set in the Conley-Morse graph $G_i$ if there does not exist an index pair $(P',E')$ for a Morse set in the Conley-Morse graph $G_{i-1}$ where $(P' \setminus E') \cap (P \setminus E) \neq \emptyset$. Similarly, the algorithm only ceases to consider a sequence ending with the index pair $(P,E)$ for a Morse set in the Conley-Morse graph $G_i$ if there does not exist an index pair $(P',E')$ for a Morse set in the Conley-Morse graph $G_{i+1}$ where $(P \setminus E) \cap (P' \setminus E') \neq \emptyset$. In addition, note that the algorithm only appends the index pair $(P,E)$ to a sequence ending in $(P',E')$ if $(P \setminus E ) \cap (P' \setminus E') \neq \emptyset$. Hence, each sequence that is constructed by the algorithm is a maximal sequence. 

It remains to be shown that the algorithm constructs all possible maximal sequences. Assume there existed some maximal sequence $\{ (P_i, E_i) \}_{i=a}^b$ that was not included in $all\_seqs$ before they are converted to filtrations. In such a case, there must exist some $(P_k,E_k)$ which was not appended to the sequence $\{ (P_i, E_i) \}_{i = a}^{k-1}$. But by inspection, we see that $(P_k,E_k)$ is appended to all such sequences so long as $(P_k \setminus E_k) \cap (P_{k-1} \setminus E_{k-1}) \neq \emptyset$. Hence, this cannot be the case, which implies that $\{(P_i,E_i)\}_{i = a}^b$ is included in $all\_seqs$.
\end{proof}

The set of Conley-Morse filtrations and the associated barcodes for the Morse decompositions in Figure \ref{fig:changing-dynamical-system} are shown in Figure \ref{fig:new-approach}.
\section{Graph Filtrations}
\label{sec:graphstructure}

Now, we show how to find a zigzag filtration that corresponds to the changing structure of the Conley-Morse graph called a \emph{graph filtration}. We begin by reviewing some properties of multivector fields. 

\begin{proposition}[See e.g. \cite{DJKKLM19,Mr2017}]
Let $\mv{1}$ and $\mv{2}$ denote multivector fields over $K$, and let $\mv{2} \sqsubseteq \mv{1}$ (see Definition~\ref{def:multivector}). Then for all $\sigma \in K$, $F_{\mv{2}}(\sigma) \subseteq F_{\mv{1}}(\sigma)$. 
\label{prop:dynamics-generator-subset}
\end{proposition}
Proposition \ref{prop:path-subset} follows directly from Proposition \ref{prop:dynamics-generator-subset}.
\begin{proposition}
Let $\mv{1}$ and $\mv{2}$ denote multivector fields over $K$, and let $\mv{2} \sqsubseteq \mv{1}$. If there exists a path $\rho$ in $K$ under $\mv{2}$ from $\sigma_0$ to $\sigma_n$, then $\rho$ is a path under $\mv{1}$.
\label{prop:path-subset}
\end{proposition}
From these two results, we deduce an important property of Morse sets.
\begin{proposition}
Let $M_1$ denote an isolated invariant set in $N$ under $\mv{1}$ and $M_2$ denote a minimal Morse set in $N$ under $\mv{2}$, where $\mv{2} \sqsubseteq \mv{1}$. If $M_1 \cap M_2 \neq \emptyset$, then $M_2 \subseteq M_1$.
\label{prop:mdsub}
\end{proposition}
\begin{proof}
Consider $\sigma \in M_1 \cap M_2$, and $\tau \in M_2$. By Proposition \ref{prop:min}, there exists a path $p$ in $N$  under $\mv{2}$ from $\sigma$ to $\tau$ and a path $q$ in $N$ under $\mv{2}$ from $\tau$ to $\sigma$. By concatenating the paths, we get a new path $r$ in $N$ under $\mv{2}$ that starts and ends at $\sigma\in M_1$. By Proposition \ref{prop:path-subset}, $r$ is a path under $\mv{2}$. But, $M_1$ is isolated by $N$, so $\im( r ) \subseteq M_1$. Hence, $\tau \in M_1$, so $M_2 \subseteq M_1$.
\end{proof}

Proposition \ref{prop:mdsub} is the starting point for constructing a zigzag filtration. For tractability, we focus on two Conley-Morse graphs $G_1$ and $G_2$, which correspond to Morse decompositions $\md{1}$ and $\md{2}$ for the isolated invariant sets $S_1$ and $S_2$ in $N$ under $\mv{1}$ and $\mv{2}$. It is natural to consider the Conley-Morse graph corresponding to the minimal Morse decomposition of $\inv(N)$ under $\mvint{1}{2}$. We call this graph $G_{1,2}$. There is a partial function $\iota_1 \; : \; V(G_{1,2}) \to V(G_1)$, where $\iota_1( M_{1,2} ) = M_1$ if $M_{1,2} \subseteq M_1$ (here, we use $M_1$ and $M_{1,2}$ to denote both a Morse set and its corresponding vertex in the Conley-Morse graph). There is a similar partial function $\iota_2 \; : \; V(G_{1,2}) \to V(G_{2})$. Note that $\iota_1$ and $\iota_2$ are not edge preserving. If there exists a directed edge from vertex $M_{1,2} \in G_{1,2}$ to vertex $M_{1,2}' \in G_{1,2}$, then there is a connection from $M_{1,2}$ to $M_{1,2}'$. By Proposition \ref{prop:path-subset}, there must exist a path from $\iota_1(M_{1,2})$ to $\iota_1(M_{1,2}')$ (likewise for $\iota_2(M_{1,2})$ to $\iota_2(M'_{1,2})$). But this path may intersect some other invariant set $M_1 \in \md{1}$, so there could exist an edge from $M_{1,2}$ to $M'_{1,2}$, but no edge from $\iota_1(M_{1,2})$ to $\iota_1(M'_{1,2})$. This phenomenon occurs in Figure \ref{fig:spurious-pers}. Hence, we need a slightly modified notion of a Conley-Morse graph under $\mvint{1}{2}$.
\begin{figure}[htpb]
\centering
\begin{tabular}{ccc}
  \includegraphics[width=42mm]{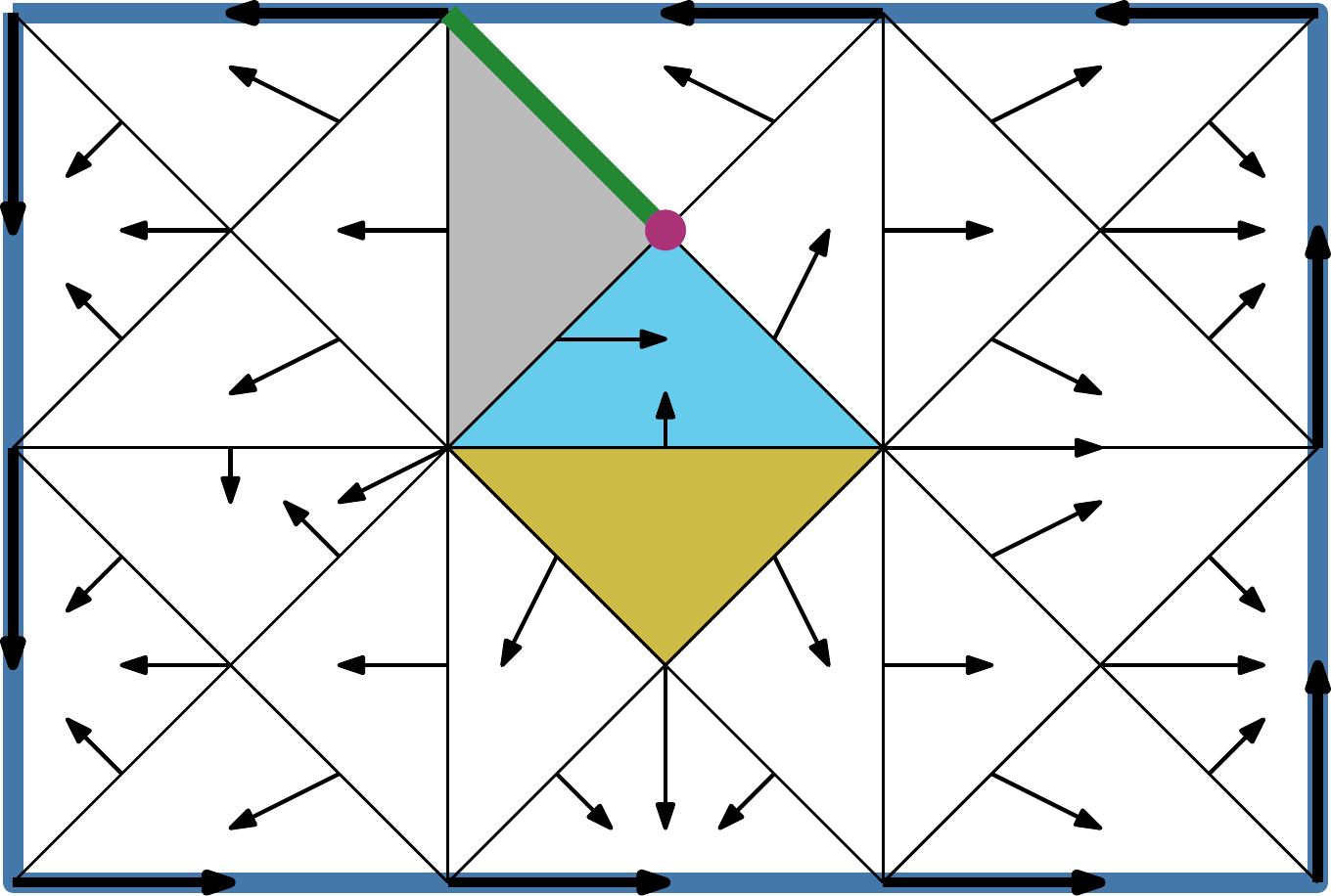}&
  \includegraphics[width=42mm]{fig/example-inter.pdf}&
  \includegraphics[width=42mm]{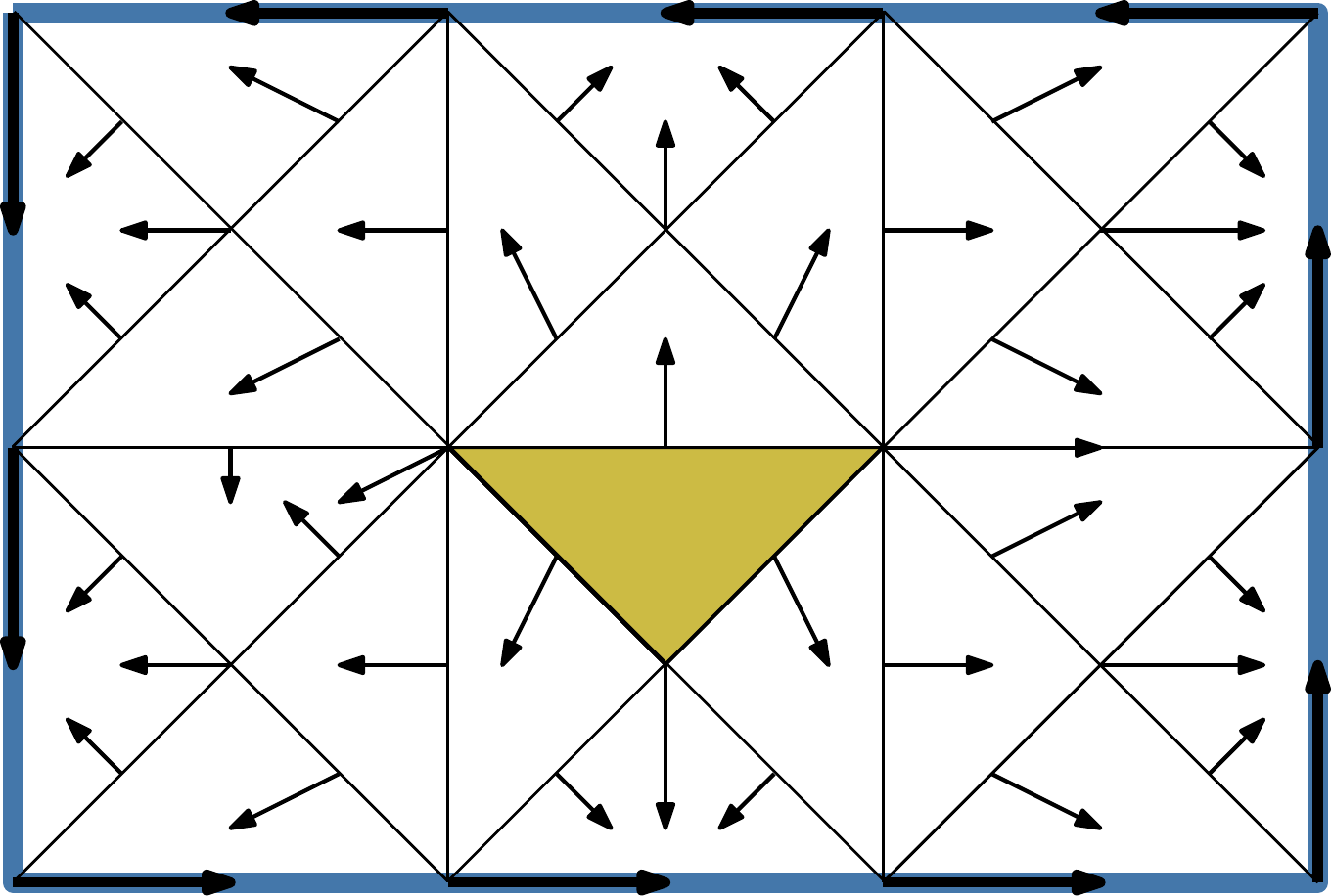}
\end{tabular}
\caption{If the multivector field on the left is $\mv{1}$ and the one on the right is $\mv{2}$, then the multivector field in the middle is $\mvint{1}{2}$. The maximal invariant set is the entire rectangle in all three multivector fields, and the Morse sets in the minimal Morse decomposition for this isolated invariant set are colored. In the middle, the Morse sets in gray, green, magenta, and pink are spurious because they are not contained in a Morse set on the right. Also in the middle, there is a connection from the golden triangle to the magenta vertex. But on the left, the connection passes through the turquoise triangle. Hence, the golden triangle-to-magenta vertex connection is not relevant.}
\label{fig:spurious-pers}
\end{figure} 
\begin{definition}[Relevant and Spurious]
Let $\md{1}$,$\md{2}$ denote Morse decompositions of isolated invariant sets $S_1$,$S_2$ in $N$ under $\mv{1}$,$\mv{2}$, and let $\md{1,2}$ denote the minimal Morse decomposition of $S_{1,2} = \inv(N)$ under $\mvint{1}{2}$. A Morse set $M_{1,2} \in \md{1,2}$ is \emph{relevant} if there exists an $M_1 \in \md{1}$ and an $M_2 \in \md{2}$ where $M_{1,2} \subseteq M_1, M_2$. If not, then $M_{1,2}$ is \emph{spurious}. 
\end{definition}

\begin{definition}[Relevant Connection]
Let $\md{1}$,$\md{2}$ denote Morse decompositions of isolated invariant sets $S_1$,$S_2$ in $N$  under $\mv{1}$,$\mv{2}$, and let $\md{1,2}$ denote the minimal Morse decomposition of $\inv(N)$ under $\mvint{1}{2}$. Also, let $\rho\; : \; \mathbb{Z} \cap [a,b] \to N$ denote a connection from $M_{1,2} \in \md{1,2}$ to $M_{1,2}' \in \md{1,2}$ under $\mv{1,2}$ where $M_{1,2}$ and $M_{1,2}'$ are relevant Morse sets. If $\rho$ satisfies both of the following:
\begin{enumerate}
    \item If $i \in \mathbb{Z} \cap [a,b]$ satisfies $\rho(i) \in M_1 \in \md{1}$, then $M_1 = \iota_1( M_{1,2} )$ or $M_1 = \iota_1( M'_{1,2} )$. 
    \item If $i \in \mathbb{Z} \cap [a,b]$ satisfies $\rho(i) \in M_2 \in \md{2}$, then $M_2 = \iota_2( M_{1,2} )$ or $M_2 = \iota_2( M'_{1,2} )$.
\end{enumerate}
then $\rho$ is a relevant connection. 
\end{definition}

We use the notions of relevant Morse sets to define the \emph{relevant Conley-Morse graph}. 

\begin{definition}[Relevant Conley-Morse Graph] 
Let $\md{1}$,$\md{2}$ denote Morse decompositions of isolated invariant sets $S_1$,$S_2$ in $N$ under $\mv{1}$,$\mv{2}$, and let $\md{1,2}$ denote the minimal Morse decomposition of $S_{1,2} = \inv(N)$ under $\mvint{1}{2}$. The \emph{relevant Conley-Morse graph} is the graph $G_{1,2}$ given by including a vertex in $V(G_{1,2})$ for each relevant Morse set in $\md{1,2}$, and including a directed edge from the vertex corresponding to $M_{1,2}$ to the vertex corresponding to $M_{1,2}'$ if there is a relevant connection from $M_{1,2}$ to $M_{1,2}'$.
\end{definition}
An example of these concepts can be seen in Figure \ref{fig:cmg-pers}. The top three graphs in Figure \ref{fig:cmg-pers} are the Conley-Morse graphs, omitting the Pointcar\'e polynomials, for the minimal Morse decompositions in Figure \ref{fig:spurious-pers}. If the top left and top right Conley-Morse graphs are Conley-Morse graphs under $\mv{1}$ and $\mv{2}$, then the top center Conley-Morse graph is a Conley-Morse graph under $\mvint{1}{2}$. The bottom left graph and bottom right graphs are the same as the top left and top right graphs, but the bottom center graph is the relevant Conley-Morse graph that is extracted from the top center graph. Each colored vertex represents the Morse set of the same color in Figure \ref{fig:spurious-pers}. Red arrows between the top center and top left or top right graphs indicate that a Morse set represented by a vertex in the top center graph is contained in a Morse set represented by a vertex in the top left or top right graphs. Hence, a vertex in the top center Conley-Morse graph is only relevant if there is a red arrow from it to a vertex in the top left and top right graphs. Relevant connections in the top center graph are shown in blue. There is a connection from the golden triangle to the blue periodic attractor by heading directly south in all three multivector fields in Figure \ref{fig:spurious-pers}. Hence, this is a relevant connection, and the edge from the golden vertex to the blue vertex is included in the relevant Conley-Morse graph in the bottom center. There is a path from the golden triangle to the magenta vertex in both the left and the center vector fields in Figure \ref{fig:spurious-pers}, but in the left multivector field, the path first passes through the turquoise Morse set. So while the path exists in both the left and middle multivector fields, it is not direct in the left multivector field, so it is not a relevant connection. Similarly, while a connection from the gray Morse set to the green Morse set exists in both the left and center multivector fields, neither the gray nor the green Morse sets are contained in a Morse set in the right multivector field, so this is not a relevant connection. Green edges represent paths which are in both the top center and top left graphs, but there is not corresponding path in the top right graph.

\begin{figure}[htbp]
\centering
  \includegraphics[width=110mm]{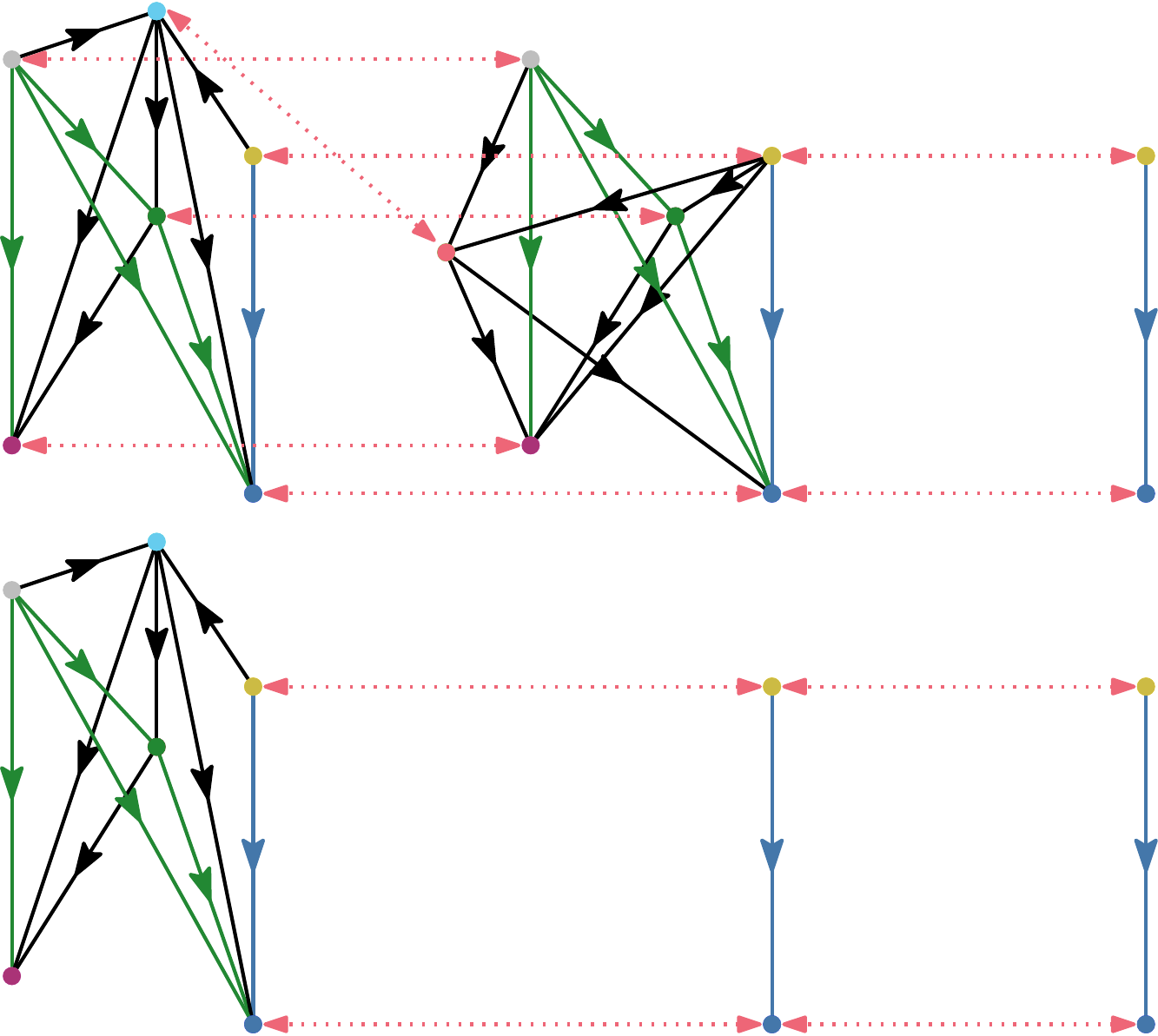}
\caption{On the top row, we depict the Conley-Morse graphs (absent the Pointcar\'e polynomials) for the Morse decompositions in Figure \ref{fig:spurious-pers}. The bottom row depicts the Conley-Morse graphs for the left and right Morse decompositions in Figure \ref{fig:spurious-pers}, but it instead includes the relevant Conley-Morse graph in the center.}
\label{fig:cmg-pers}
\end{figure} 

For the remainder of the paper, whenever we refer to a graph $G_{i,i+1}$ for some $i$, we are referring to the relevant Conley-Morse graph under $\mvint{i}{i+1}$. Note that we compute relevant Conley-Morse graphs by taking subgraphs of a Conley-Morse graph. Hence, $\iota_1$ and $\iota_2$ restrict to functions $\iota_1 \; : \; G_{1,2} \to G_1$ and $\iota_2 \; : \; G_{1,2} \to G_2$ (they are no longer partial).
\begin{proposition}
Let $G_1$ and $G_2$ denote the Conley-Morse graphs for Morse decompositions of the isolated invariant sets $S_1$,$S_2$ in $N$ under $\mv{1}$,$\mv{2}$, and let $G_{1,2}$ denote the relevant Conley-Morse graph for the minimal Morse decomposition of the maximal invariant set in $N$ under $\mvint{1}{2}$. If there is a directed edge from $M_{1,2}$ to $M_{1,2}'$ in $G_{1,2}$, then either $\iota_k( M_{1,2} ) = \iota_k( M'_{1,2})$ or there exists a directed edge from $\iota_k( M_{1,2} )$ to $\iota_k( M'_{1,2} )$ for $k \in \{1,2\}$. 
\label{prop:edgepreserving}
\end{proposition}
\begin{proof}
If there exists a directed edge from $M_{1,2}$ to $M'_{1,2}$, then there exists a relevant connection from $M_{1,2}$ to $M'_{1,2}$ under $\mvint{1}{2}$. We denote this connections as $\rho \; : \; [0,n] \cap \mathbb{Z} \to N$. By Proposition \ref{prop:path-subset}, $\rho$ must also be a path under $\mv{1}$ and $\mv{2}$. By the definition of relevant connection, if $\rho(i) \in M \in \md{1}$, then $\iota_1( M_{1,2} ) = M$ or $\iota_1(M'_{1,2}) = M$. Hence, either $\iota_1( M_{1,2} ) = \iota_1(M'_{1,2})$, or as $\rho$ is a direct connection from $\iota_1( M_{1,2} )$ to $\iota_1(M'_{1,2})$, there exists a directed edge from the former to the later. The same argument holds for $\iota_2$. 
\end{proof}

Now, we have two Conley-Morse graphs $G_1$,$G_2$ for isolated invariant sets in $N$ under $\mv{1}$,$\mv{2}$ and the relevant Conley-Morse graph $G_{1,2}$ for the maximal invariant set in $N$ under $\mvint{1}{2}$. We are interested in how the structure of these graphs ``persist.'' To do this, we will treat the graphs as simplicial complexes by ignoring the orientation on edges. The maps $\iota_1$,$\iota_2$ induce simplicial maps on these complexes.
\begin{proposition}
Let $f_1 \; : \; G_{1,2} \to G_1$ and $f_2 \; : \; G_{1,2} \to G_2$ denote the maps induced by $\iota_1$, $\iota_2$ where $f_1(\{u,v\}) = \{\iota_1(u), \iota_1(v)\}$ and $f_2(\{u,v\}) = \{\iota_2(u), \iota_2(v)\}$. The maps $f_1$ and $f_2$ are simplicial maps.
\label{prop:simpmap}
\end{proposition}
\begin{proof}
The maps $f_1$ and $f_2$ bring vertices to vertices, so it is sufficient to show that the image of an edge is either an edge or a vertex. If $\sigma \in G_{1,2}$ is an edge, then it corresponds to an edge $e = (u,v) \in G_{1,2}$. Such edges correspond to relevant paths from the vertex $u$ to the vertex $v$. Thus, either $\iota_1( u) = \iota_1(v)$, which implies that $f_1( \{u,v\} ) = \iota_1(u)$, or $\iota_1(u) \neq \iota_1(v)$, which implies that there is a connection from $\iota_1(u)$ to $\iota_1(v)$. Hence, there exists an edge $\{\iota_1(u), \iota_1(v) \} \in G_1$. Thus, $f_1$ is both a map and a simplicial map. The argument for $f_2$ follows analogously.
\end{proof}

Given a sequence of $n$ Conley-Morse graphs and $n-1$ relevant Conley-Morse graphs, we can use Proposition \ref{prop:simpmap} to obtain a sequence of complexes connected by simplicial maps. We show this sequence in Equation \ref{eqn:zzfilt}:
\begin{equation}
    G_1 \leftarrow G_{1,2} \rightarrow G_2 \leftarrow G_{2,3} \rightarrow \cdots \leftarrow G_{n-1,n} \rightarrow G_n.
    \label{eqn:zzfilt}
\end{equation}
Hence, we have a zigzag filtration that captures the changing structure of a sequence of Conley-Morse graphs. We can compute the barcode for zigzag filtrations where the maps are simplicial by using an algorithm in \cite{DFW2014}.

\section{Barcodes for Conley-Morse Graphs}
\label{sec:barcode}

In Section \ref{sec:persist} we showed how to find Conley-Morse filtrations, which represent the changing Conley indices at the vertices of a Conley-Morse graph. Similarly, in Section \ref{sec:graphstructure}, we showed how to extract a graph filtration, which represents the changing graph structure of the Conley-Morse graph. By taking all barcodes for these filtrations together, we can straightforwardly obtain a ``barcode'' for a sequence of Conley-Morse graphs. However, the barcode obtained from Conely-Morse filtrations may contain a significant amount of redundancy. Recall that we compute the changing Conley index by considering all possible Conley-Morse filtrations given by maximal sequences $\{ (P_i, E_i) \}_{i = j}^k$ where $(P_i, E_i)$ is an index pair in $N$ for a Morse set in the $i$th Conley-Morse graph and $(P_i \setminus E_i) \cap (P_{i+1} \setminus E_{i+1}) \neq \emptyset$. As a result of this construction, a subfiltration given by the sequence $\{ (P_i, E_i) \}_{i = j}^k$ may occur in several filtrations.

Duplication of subfiltrations can lead to a duplication of bars. For an example, in Figure \ref{compl-fig}, consider the middle and bottom filtrations. The middle filtration shows a repelling fixed point turning into a periodic repeller which in turn merges with a periodic attractor to become a semistable limit cycle. In the bottom filtration, a repelling fixed point becomes an attracting fixed point. The $2$-dimensional barcode for the middle filtration is given by a single bar, which represents the $2$-dimensional homology generator for the repelling fixed point and the periodic repeller. In contrast, the $2$-dimensional bar code for the bottom filtration represents only the lifetime of the repelling fixed point. This is redundant: the  $2$-dimensional bar for the middle filtration already captures the homology generator for the repelling fixed point in the bottom filtration. One bar is a subset of the other, and they capture the same generator.

\begin{figure}[htbp]
\centering
\begin{tabular}{ccc}
  \includegraphics[height=32mm]{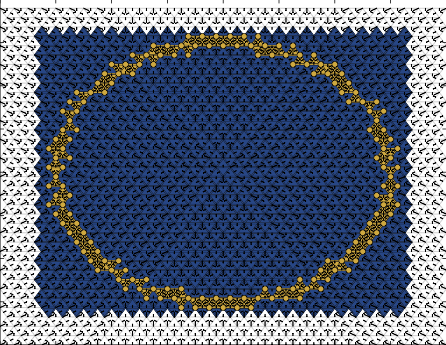}&
  \includegraphics[height=32mm]{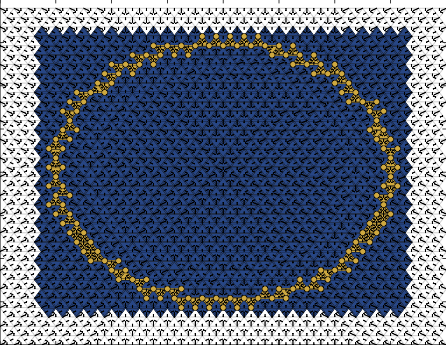}&
  \includegraphics[height=32mm]{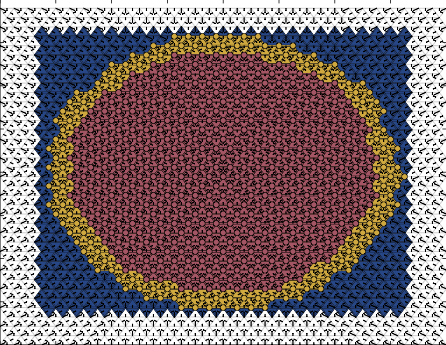}\\
  \multicolumn{3} {l} {
  \begin{tikzpicture}[outer sep = 0, inner sep = 0]
  \draw[fill=white,outer sep=0, inner sep=0] (0.0,-0.3) -- (9.08,-0.3) -- (9.08,0.1) -- (0.0,0.1) -- cycle;
  \node[align=left] at (1.3,-0.1) {Dimension: 0};
  \draw[fill=light-gray,outer sep=0, inner sep=0] (0.0,-0.85) -- (9.08,-0.85) -- (9.08,-0.45) -- (0.0,-0.45) -- cycle;
  \node[align=left] at (1.3,-0.65) {Dimension: 1};
  \end{tikzpicture} }\\
  \includegraphics[height=32mm]{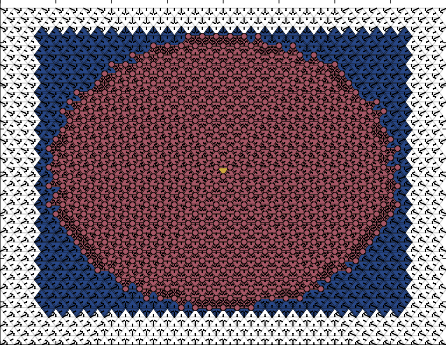}&
  \includegraphics[height=32mm]{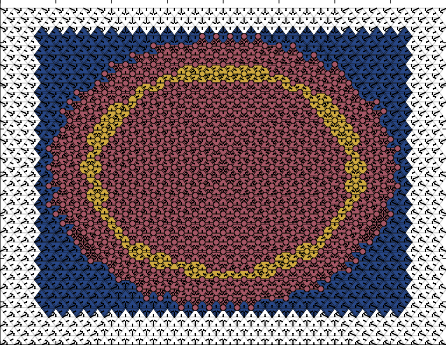}&
  \includegraphics[height=32mm]{fig/high-contrast-negative-repeller-220.png}\\
  \multicolumn{3} {c} {
  \begin{tikzpicture}[outer sep = 0, inner sep = 0]
  \draw[fill=light-gray,outer sep=0, inner sep=0] (0.0,-0.3) -- (4.84,-0.3) -- (4.84,0.1) -- (0.0,0.1) -- cycle;
  \node[align=left] at (1.3,-0.1) {Dimension: 1};
  \end{tikzpicture} }\\
  \multicolumn{3} {l} {
  \begin{tikzpicture}[outer sep = 0, inner sep = 0]
  \draw[fill=dark-gray,outer sep=0, inner sep=0] (0.0,-0.3) -- (9.08,-0.3) -- (9.08,0.1) -- (0.0,0.1) -- cycle;
  \node[align=left,color=white] at (1.3,-0.1) {Dimension: 2};
  \end{tikzpicture} }\\
  \includegraphics[height=32mm]{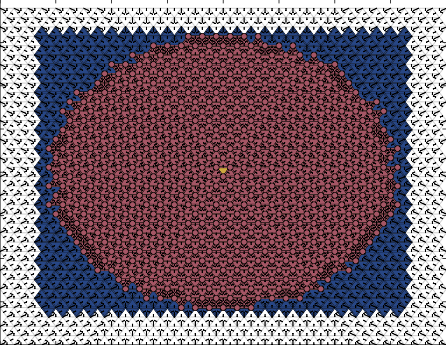} &
  \includegraphics[height=32mm]{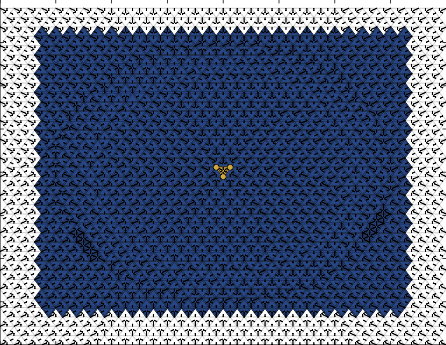}&
  \includegraphics[height=32mm]{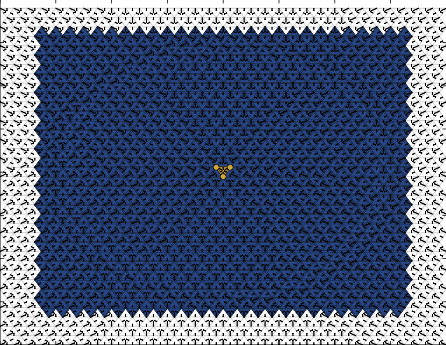}\\
  \multicolumn{3} {l} {
  \begin{tikzpicture}[outer sep = 0, inner sep = 0]
  \draw[fill=dark-gray,outer sep=0, inner sep=0] (0.0,-0.3) -- (4.14,-0.3) -- (4.14,0.1) -- (0.0,0.1) -- cycle;
  \node[align=left,color=white] at (1.3,-0.1) {Dimension: 2};
  \end{tikzpicture} }\\
  \multicolumn{3} {r} {
  \begin{tikzpicture}[outer sep = 0, inner sep = 0]
  \draw[fill=white,outer sep=0, inner sep=0] (0.0,-0.3) -- (9.07,-0.3) -- (9.07,0.1) -- (0.0,0.1) -- cycle;
  \node[align=left] at (1.3,-0.1) {Dimension: 0};
  \end{tikzpicture} }\\  
\end{tabular}
\caption{All three maximal sequences extracted from the changing Conley-Morse graph for the Morse decompositions in Figure \ref{fig:changing-dynamical-system}. The isolating set is given by yellow, red, and blue simplices, while if $(P,E)$ is an index pair, the simplices in $P\setminus E$ are in yellow and the simplices in $E$ are in red. The top three images show a periodic attractor that becomes a semistable limit cycle. The middle three images show a repelling fixed point that becomes a periodic repeller and then becomes a semistable limit cycle, and the bottom three images show a repelling fixed point that transitions into an attracting fixed point. Beneath each maximal sequence, we include the barcode from the zigzag filtration that we get by applying Theorem \ref{thm:interindexpair} to the maximal sequence to obtain a Conley-Morse filtration. White bars are $0$-dimensional, light gray bars are $1$-dimensional, and dark gray bars are $2$-dimensional. }
\label{compl-fig}
\end{figure} 

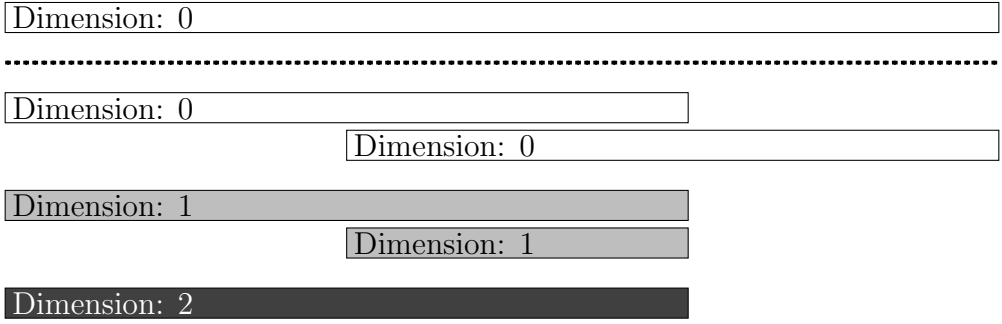
\begin{figure}
\centering
\begin{tikzpicture}[outer sep = 0, inner sep = 0]
  \draw[fill=white,outer sep=0, inner sep=0] (0.0,0.9) -- (13.21,0.9) -- (13.21,1.3) -- (0.0,1.3) -- cycle;
  \node[align=left] at (1.3,1.1) {Dimension: 0};
  \draw[line width = .5mm, dotted](0, .5) -- (13.21,.5) -- cycle;
  \draw[fill=white,outer sep=0, inner sep=0] (0.0,-0.3) -- (9.08,-0.3) -- (9.08,0.1) -- (0.0,0.1) -- cycle;
  \node[align=left] at (1.3,-0.1) {Dimension: 0};
  \draw[fill=white,outer sep=0, inner sep=0] (4.54,-0.8) -- (13.21,-0.8) -- (13.21,-0.4) -- (4.54,-0.4) -- cycle;
  \node[align=left] at (5.84,-0.6) {Dimension: 0};
  \draw[fill=light-gray,outer sep=0, inner sep=0] (0.0,-1.6) -- (9.08,-1.6) -- (9.08,-1.2) -- (0.0,-1.2) -- cycle;
  \node[align=left] at (1.3,-1.4) {Dimension: 1};
  \draw[fill=light-gray,outer sep=0, inner sep=0] (4.54,-2.1) -- (9.08,-2.1) -- (9.08,-1.7) -- (4.54,-1.7) -- cycle;
  \node[align=left] at (5.84,-1.9) {Dimension: 1};
  \draw[fill=dark-gray,outer sep=0, inner sep=0] (0.0,-2.9) -- (9.08,-2.9) -- (9.08,-2.5) -- (0.0,-2.5) -- cycle;
  \node[align=left,color=white] at (1.3,-2.7) {Dimension: 2};
\end{tikzpicture}
\caption{The barcode corresponding to the sequence of combinatorial dynamical systems in Figure \ref{fig:changing-dynamical-system}. The white bar above the dotted line is $0$-dimensional, and it represents the connected component in the Conley-Morse graph. The bars below the dotted line are obtained by extracting the barcodes from the Conley-Morse filtrations in Figure \ref{compl-fig} and removing redundant bars.}
\label{fig:barcode}
\end{figure}

In devising a barcode that captures a changing Conley-Morse graph, it is best that we eliminate these redundant bars. In this section, we use the notation $\mathcal{I}_{a,b}$ to denote an interval $I$ such that $\cl(I) = [a,b]$. That is, $\mathcal{I}_{a,b}$ can denote any of $[a,b]$, $[a,b)$, $(a,b]$, or $(a,b)$. We also consider \emph{subfiltrations}. If $\mathcal{F}$ denotes the zigzag filtration $(P_1, E_1) \supseteq (P_1 \cap P_2, E_1 \cap E_2 ) \subseteq (P_2, E_2) \supseteq \ldots$, then we use the notation $\mathcal{F}_{a,b}$ to denote the subfiltration $(P_a, E_a) \supseteq (P_a \cap P_{a+1}, E_a \cap E_{a+1} ) \subseteq (P_{a+1}, E_{a+1}) \supseteq \ldots \subseteq (P_b, E_b)$. 
\begin{definition}[Redundant]
Let $\mathcal{Z}$ denote the set of all maximal relative zigzag filtrations for a sequence of Conley-Morse graphs, and let $\mathcal{I}_{a,b}$ denote a $k$-dimensional bar extracted from $\mathcal{F} \in \mathcal{Z}$. If there exists a filtration $\mathcal{G} \in \mathcal{Z}$ where $\mathcal{G}_{a,b} = \mathcal{F}_{a,b}$, then the bar $\mathcal{I}_{a,b}$ is \emph{redundant}. 
\end{definition}
\begin{proposition}
Let $\mathcal{Z}$ denote the set of maximal relative zigzag filtrations for a sequence of Conley-Morse graphs. If $\mathcal{I}_{a,b}$ is a redundant $k$-dimensional bar in the barcode for $\mathcal{F} \in \mathcal{Z}$, then there exists a filtration $\mathcal{G}\in \mathcal{Z}$ such that the $k$-dimensional barcode for $\mathcal{G}$ contains a bar $\mathcal{I}_{c,d}$ where $\mathcal{I}_{a,b} \subseteq \mathcal{I}_{c,d}$.
\end{proposition}
\begin{proof}
By restricting the $k$-dimensional barcode for $\mathcal{G}$ and $\mathcal{F}$ to the interval $[a,b]$, we get two sets of bars for the zigzag filtration $\mathcal{F}_{a,b} = \mathcal{G}_{a,b}$. The barcode for $\mathcal{F}_{a,b} = \mathcal{G}_{a,b}$ is unique (see \cite{zigzag}). Hence, there must be a bar $\mathcal{I}_{c,d}$ in the barcode for $\mathcal{G}$ where $\mathcal{I}_{c,d} \cap [a,b] = \mathcal{I}_{a,b}$.
\end{proof}

If $\mathcal{Z}$ is the set of zigzag filtrations for the Conley-Morse graph, and $\mathcal{B}$ denotes the set of bars extracted from filtrations in $\mathcal{Z}$, then redundancy gives a partial order on $\mathcal{B}$. In particular, $\mathcal{I}_{a,b} \leq \mathcal{I}_{c,d}$ if $\mathcal{I}_{a,b} \subseteq \mathcal{I}_{c,d}$. Together with the bars extracted from the graph structure filtration in Section \ref{sec:graphstructure}, the set of maximal bars under $\leq$ are taken to be the barcode for the changing Conley-Morse graph. 

\begin{algorithm2e}
\SetAlgoLined
\KwIn{ List of filtrations $\mathcal{F}_1, \ldots, \mathcal{F}_n$, and the corresponding set of $k$-dimensional barcodes $\mathcal{B}(\mathcal{F}_1), \ldots, \mathcal{B}(\mathcal{F}_n)$} 
\KwOut{ Set of bars $\mathcal{B}$ }
    
    $\mathcal{B} \gets \texttt{new set}()$
    
    \For{$i \in 1, \ldots, n$ }
    {
        $\mathcal{F} \gets \mathcal{F}_i$
        
        \For{$\mathcal{I}_{a,b} \in \mathcal{B}(\mathcal{F}_i)$}
        {
            $redundant \gets \texttt{False}$
        
            \For{$j \in 1 \ldots n$}
            {
                \If{$j \neq i$}
                {
                    
                    $\mathcal{G} \gets \mathcal{G}_j$
                
                    \If{$!\texttt{forbidden}(\mathcal{G}_{a,b}) \; \texttt{and} \;
                    \mathcal{F}_{a,b} = \mathcal{G}_{a,b}$}
                    {
                        $redundant \gets \texttt{True}$
                        
                        $\texttt{forbidden}( \mathcal{F}_{a,b}) \gets \texttt{True}$
                    }
                }
            }
            
            \uIf{ $redundant = \texttt{False}$ }
            {
                $\texttt{add}(\mathcal{B}, \mathcal{I}_{a,b})$
            }\Else{
                $\mathcal{B}(\mathcal{F}_i) \gets \mathcal{B}(\mathcal{F}_i) \setminus \{ \mathcal{I}_{a,b} \}$
            }
        }
    }
    
    \Return{$\mathcal{B}$}
    
 \caption{$\texttt{EliminateRedundancies}( \{ \mathcal{F}_i \}_{i = 1}^n, \{\mathcal{B}(\mathcal{F}_i)\}_{i=1}^n)$ }
 \label{alg:redunds}
\end{algorithm2e}

\begin{proposition}
Algorithm \ref{alg:redunds} outputs the set of maximal bars.
\end{proposition}
\begin{proof}
Consider the set of maximal bars, denoted $\mathcal{B}_M$. We consider two cases. In the first, the maximal bar $B_{a,b}$ corresponds to a zigzag subfiltration in exactly one zigzag filtration $\mathcal{F}$. In such a case, there is no other filtration $\mathcal{G}$ where $\mathcal{F}_{a,b} = \mathcal{G}_{a,b}$, so when considering $B_{a,b}$ the redundant flag will not be \texttt{True} and hence $B_{a,b}$ will be added to $\mathcal{B}$. Now, assume that the maximal bar $\mathcal{B}_{a,b}$ corresponds to the same subfiltration in $\mathcal{F}_{i_1}$, \ldots, $\mathcal{F}_{i_k}$, where $i_1 \leq \ldots \leq i_k$. For each filtration $\mathcal{F}_{i_j}$, $i_j < i_k$, we claim that $B_{a,b}$ will be considered redundant. Clearly this is the case, because $B_{a,b}$  corresponds to a subfiltration $\mathcal{F}_{i_k,a,b}$, and as $\mathcal{F}_{i_k}$ has not been processed at the time that $\mathcal{F}_{i_j}$ is processed for $i_j < i_k$, it follows that when each is processed and $B_{a,b}$ is considered, the redundant flag will be set to \texttt{True} and $B_{a,b}$ will not be added to $\mathcal{B}$. However, when $\mathcal{F}_{i_k}$ is considered, all of the previous $\mathcal{F}_{i_j,a,b}$ will have been marked forbidden, and $B_{a,b}$ will be added to to $\mathcal{B}$. Thus, $B_{a,b}$ will be added exactly once. 

Hence, all maximal bars are included in $\mathcal{B}$. Note that every non-maximal bar $B_{a,b}$ will be marked redundant, because if $B_{a,b}$ is non-maximal then it corresponds to some filtration $\mathcal{F}$ and there must exist some maximal bar $B'_{c,d}$ corresponding to the filtration $\mathcal{G}$ where $\mathcal{G}_{a,b} = \mathcal{F}_{a,b}$. Since $B'_{c,d}$ is maximal, we have already established that at least one copy of it will not be marked redundant, so $\mathcal{F}_{a,b}$ can be compared against $\mathcal{G}_{a,b}$. Hence, $B_{a,b}$ will be marked redundant. 
\end{proof}

\section{Choosing Index Pairs}
\label{sec:multinode}

In Section \ref{sec:persist}, we explicated a technique for extracting a set of Conley Morse filtrations from a sequence of Conley-Morse graphs $G_1$, $G_2$, \ldots, $G_m$. Each Conley-Morse graph corresponds to a Morse decomposition $\md{1}$, $\md{2}$, \ldots, $\md{m}$. Crucially, our technique requires a fixed index pair in $N$, denoted $(P,E)$, for each Morse set $M \in \md{i}$. There are multiple approaches to choosing an index pair for $M$. A natural starting point is in Proposition \ref{prop:ip} and Proposition \ref{prop:pfip}, which together imply that $(\pf_N( \cl( M )), \pf_N( \mo( M ) ))$ is an index pair for $M$ in $N$. However, this approach can be problematic. Consider Figure \ref{fig:unthickened}. In the first and third images, we find an index pair for the isolated invariant sets $M_1$, $M_2$ by using this technique. The simplices in yellow and blue are the simplices in $N$, and the simplices in blue are in $P$. In these two cases, $E = \emptyset$. Intuitively, we expect that a $1$-dimensional homology class would persist from the left to the right. However, this is not the case. The middle image shows the index pair $(\cl(M_1) \cap \cl(M_2), \mo(M_1) \cap \mo(M_2))$ in blue, and $H_1( \cl(M_1) \cap \cl(M_2), \mo(M_1) \cap \mo(M_2)) = 0$. Thus, a one dimensional homology class does not persist from the first to third image. We devise a new method to compute index pairs which permits us to circumvent this issue. For a set of multivectors $\mathcal{A}$, we let $\langle \mathcal{A} \rangle := \cup_{A \in \mathcal{A}} A$. 

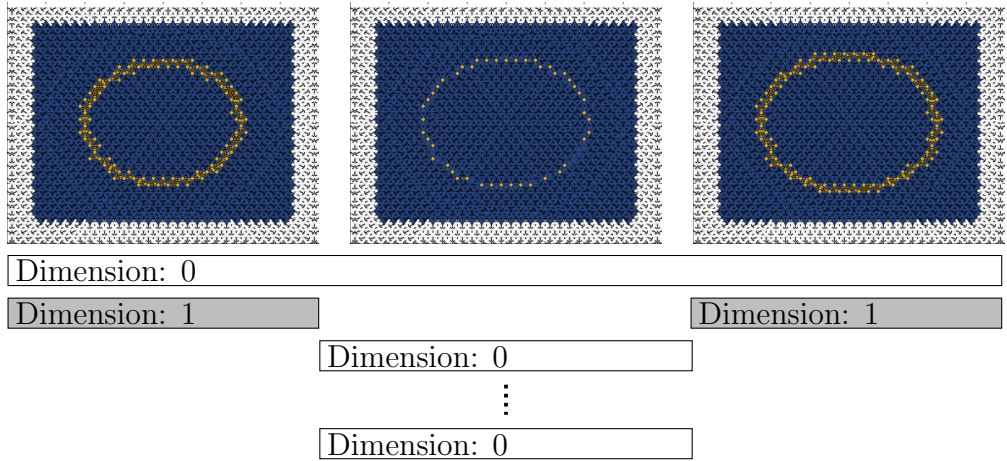
\begin{figure}[htbp]
\centering
\begin{tabular}{ccc}
  \includegraphics[height=32mm]{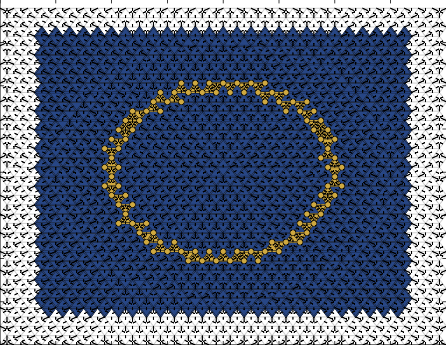}&
  \includegraphics[height=32mm]{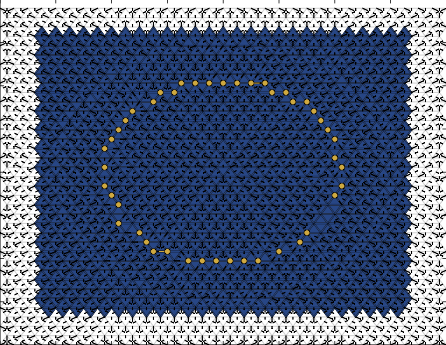}&
  \includegraphics[height=32mm]{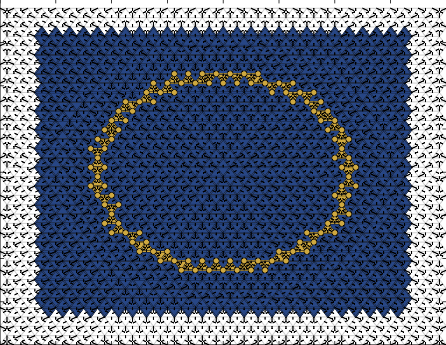}\\
  \multicolumn{3} {l} {
  \begin{tikzpicture}[outer sep = 0, inner sep = 0]
  \draw[fill=white,outer sep=0, inner sep=0] (0.0,-0.3) -- (13.21,-0.3) -- (13.21,0.1) -- (0.0,0.1) -- cycle;
  \node[align=left] at (1.3,-0.1) {Dimension: 0};
  \end{tikzpicture} }\\
  \multicolumn{3} {l} {
  \begin{tikzpicture}[outer sep = 0, inner sep = 0]
  \draw[fill=light-gray,outer sep=0, inner sep=0] (0.0,-0.3) -- (4.13,-0.3) -- (4.13,0.1) -- (0.0,0.1) -- cycle;
  \node[align=left] at (1.3,-0.1) {Dimension: 1};
  \draw[fill=light-gray,outer sep=0, inner sep=0] (9.08,-0.3) -- (13.21,-0.3) -- (13.21,0.1) -- (9.08,0.1) -- cycle;
  \node[align=left] at (10.38,-0.1) {Dimension: 1};
  \end{tikzpicture} }\\
  \multicolumn{3} {c} {
  \begin{tikzpicture}[outer sep = 0, inner sep = 0]
  \draw[fill=white,outer sep=0, inner sep=0] (0.0,-0.3) -- (4.95,-0.3) -- (4.95,0.1) -- (0.0,0.1) -- cycle;
  \node[align=left] at (1.3,-0.1) {Dimension: 0};
  \end{tikzpicture} }\\
  \multicolumn{3} {c} {
  \begin{tikzpicture}[outer sep = 0, inner sep = 0]
  \draw[line width = .5mm, dotted](0, 0) -- (0,-0.4);
  \end{tikzpicture} }\\
  \multicolumn{3} {c} {
  \begin{tikzpicture}[outer sep = 0, inner sep = 0]
  \draw[fill=white,outer sep=0, inner sep=0] (0.0,-0.3) -- (4.95,-0.3) -- (4.95,0.1) -- (0.0,0.1) -- cycle;
  \node[align=left] at (1.3,-0.1) {Dimension: 0};
  \end{tikzpicture} }
\end{tabular}
\caption{In all three images, the isolating set $N$ is given by the blue and the yellow simplices, while an isolated invariant set in $N$ is given by the yellow simplices. In addition, in the left and the right images, we can obtain an index pair in $N$ by taking the yellow simplices to be $P$ and letting $E := \emptyset$. If $(P_1,E_1)$ denotes the index pair on the left and $(P_2, E_2)$ denotes the index pair in $N$ on the right, then the index pair in the center is given by $(P_1 \cap P_2, E_1 \cap E_2)$. The yellow simplices in the middle are those simplices in $P_1 \cap P_2$. Intuitively, one would expect a $0$-dimensional bar and a $1$-dimensional bar to persist through all three images. However, when one takes the intersection of the two index pairs, one no longer has an annulus, so the $1$-dimensional bar does not persist through all three images. Instead, we obtain several short, $0$-dimensional bars. We show the barcode beneath the images, while excluding several of the short $0$-dimensional bars.  }
\label{fig:unthickened}
\end{figure} 

\begin{proposition}
Let $(P,E)$ denote an index pair in $N$ under $\mathcal{V}$, and let $\mathcal{A}$ denote a set of multivectors such that $\langle \mathcal{A} \rangle \subseteq N$, $\langle \mathcal{A} \rangle \cap E = \emptyset$ and $\mo( \langle \mathcal{A} \rangle ) \subseteq P$. The pair $(P \cup \langle \mathcal{A} \rangle, E)$ is an index pair in $N$ under $\mv{}$. 
\label{prop:expand}
\end{proposition}
\begin{proof}
It is immediate that $\inv( ( P \cup \langle \mathcal{A} \rangle ) \setminus E ) = \inv( ( P \cup \langle \mathcal{A} \rangle ) \setminus E )$. In addition, note that $E$ has not changed, so $F_{\mv{}}(E) \cap N \subseteq E$ by hypothesis. Hence, it is sufficient to show that $F_{\mv{}}( P \cup \langle \mathcal{A} \rangle ) \cap N \subseteq P \cup \langle \mathcal{A} \rangle$ and $F_{\mv{}}(  P \cup \langle \mathcal{A} \rangle ) \setminus E \subseteq N$. 

First, we consider $F_{\mv{}}( \langle \mathcal{A} \rangle )$. The set $\langle \mathcal{A} \rangle$ is a union of multivectors, so by definition, $F_{\mv{}}( \langle \mathcal{A} \rangle ) = \langle A \rangle \cup \cl( \langle \mathcal{A} \rangle )$. By definition, $\cl( \langle \mathcal{A} \rangle ) = \langle \mathcal{A} \rangle \cup \mo( \mathcal{A})$. Thus, it follows that $F_{\mv{}}( \langle \mathcal{A} \rangle ) = \langle \mathcal{A} \rangle \cup \mo( \langle \mathcal{A} \rangle )$. By assumption, $\mo( \langle \mathcal{A} \rangle ) \subseteq P$, so it follows that $F_{\mv{}}( \langle \mathcal{A} \rangle ) \subseteq \langle A \rangle \cup P \subseteq N$.  

Now, we show that $F_{\mv{}}( P \cup \langle \mathcal{A} \rangle ) \cap N \subseteq P \cup \langle \mathcal{A} \rangle$. By the definition of $F$, it follows that $F_{\mv{}}( P \cup \langle \mathcal{A} \rangle ) = F_{\mv{}}( P ) \cup F_{\mv{}}( \langle \mathcal{A} \rangle )$. The pair $(P,E)$ is an index pair in $N$, so it follows that $F_{\mv{}}( P ) \cap N \subseteq P$. We have already shown that $F_{\mv{}}( \langle \mathcal{A} \rangle ) \subseteq P \cup \langle \mathcal{A} \rangle \subseteq N$, so it follows that $F_{\mv{}}( P \cup \langle \mathcal{A} \rangle ) \cap N \subseteq P \cup \langle \mathcal{A} \rangle$. 

Now, we move to showing that $F_{\mv{}}( (P \cup \langle \mathcal{A} \rangle ) \setminus E ) \subseteq N$. By assumption, $\langle \mathcal{A} \rangle \cap E = \emptyset$. Thus, $(P \cup \langle \mathcal{A} \rangle) \setminus E = (P \setminus E) \cup \langle \mathcal{A} \rangle$. Hence, $F_{\mv{}}( (P \cup \langle \mathcal{A} \rangle) \setminus E )  = F_{\mv{}}( P \setminus E) \cup F_{\mv{}}( \langle \mathcal{A} \rangle )$. The pair $(P,E)$ is an index pair in $N$, so $F_{\mv{}}( P \setminus E) \subseteq N$. We have already shown that $F_{\mv{}}( \langle \mathcal{A} \rangle ) \subseteq N$, so it follows that $F_{\mv{}}( (P \cup \langle \mathcal{A} \rangle ) \setminus E ) \subseteq N$. 

Hence, we conclude that $(P \cup \langle \mathcal{A} \rangle, E)$ is an index pair for $\inv( (P \cup \langle \mathcal{A} \rangle ) \setminus E )$ in $N$. 
\end{proof}

However, Proposition \ref{prop:expand} does not imply that $\inv( P \setminus E ) = \inv( (P \cup \langle \mathcal{A}) \setminus E)$. We strengthen Proposition \ref{prop:expand} to account for this. 

\begin{proposition}
Let $N$ denote a closed set, and let $\md{}$ denote the minimal Morse decomposition for $\inv(N)$. Furthermore, let $(P,E)$ denote an index pair in $N$ for an isolated invariant set $S$ in $N$. If $\mathcal{A}$ is a set of multivectors where $\langle \mathcal{A} \rangle \subseteq N$, $\langle \mathcal{A} \rangle \cap E = \emptyset$, $\mo( \langle \mathcal{A} \rangle ) \subseteq P$, and $\langle \mathcal{A} \rangle \cap M = \emptyset$ for $M \in \mathcal{M}$ where $M \not\subseteq S$, then $(P \cup \langle \mathcal{A} \rangle, E)$ is an index pair in $N$ for $S$.
\label{prop:thicken}
\end{proposition}
\begin{proof}
By Proposition \ref{prop:expand}, $(P \cup \langle \mathcal{A} \rangle, E)$ is an index pair for $\inv( (P \cup \langle A \rangle ) \setminus E)$, so it is sufficient to show that $\inv( ( P \cup \langle \mathcal{A} \rangle ) \setminus E ) = S$. Trivially, $S \subseteq \inv( (P \cup \langle \mathcal{A} \rangle ) \setminus E)$. Hence, we aim to show the reverse inclusion. To contradict, assume that $\sigma \in \inv( (P \cup \langle \mathcal{A} \rangle ) \setminus E ) \setminus S$. Then there exists some essential solution $\rho \; : \; \mathbb{Z} \to (P \cup \langle \mathcal{A} \rangle) \setminus E$ where $\rho( 0 ) = \sigma$. In addition, there must exist $M_1, M_2 \in \md{}$ such that $\alpha( \rho ) \subseteq M_1$ and $\omega( \rho ) \subseteq M_2$.  

We have that $\im( \rho ) \subseteq (P \cup \langle \mathcal{A} \rangle) \setminus E$, so $M_1 \cap (P \cup \langle \mathcal{A} \rangle ) \setminus E \neq \emptyset$, and similarly for $M_2$. We claim that $M_1, M_2 \subseteq (P \cup \langle \mathcal{A} \rangle ) \setminus E$. If there were a $\eta \in M_1$ but $\eta \not\in (P \cup \langle \mathcal{A} \rangle)\setminus E$, then since there is a $\tau \in M_1 \cap ( P \cup \langle \mathcal{A} )\setminus E$, Proposition \ref{prop:min} implies there exists a path from $\eta$ to $\tau$ and there exists a path from $\tau$ to $\eta$. But by the definition of an index pair in $N$, it follows that $\eta \in E$. Since there is a path from $\eta$ to $\tau$, but $\tau \in (P \cup \langle \mathcal{A} \rangle) \setminus E$, this contradicts the requirement that $F_{\mv{}}( E ) \cap N \subseteq E$. Hence, no such $\eta$ can exist, and $M_1,M_2 \subseteq (P \cup \langle \mathcal{A} \rangle) \setminus E$. 

By assumption, $\langle \mathcal{A} \rangle \cap M_1 = \emptyset$ if $M_1 \not\subseteq S$ (and similarly for $M_2$), so $M_1 \cap (\langle \mathcal{A} \rangle  \setminus P ) = \emptyset$ (and similarly for $M_2$). Thus, $M_1, M_2 \subseteq P \setminus E$. But $S = \inv( P \setminus E)$, so $M_1, M_2 \subseteq S$. But this implies there is a path from $S$ to $\sigma$ and a path from $\sigma$ to $S$, which implies that $S$ is not isolated by $N$, a contradiction. 
\end{proof}
Proposition \ref{prop:thicken} provides a natural avenue for finding index pairs that are suitable for computing the persistence of the Conley-Morse graph. If we are given an index pair $(P,E)$ in $N$ for a Morse set $M$, we can incrementally add multivectors to $P$ that satisfy the conditions of Proposition \ref{prop:thicken}, up to a specified distance away from the original $P$. We give an example of this in Figure \ref{fig:thickened}. 

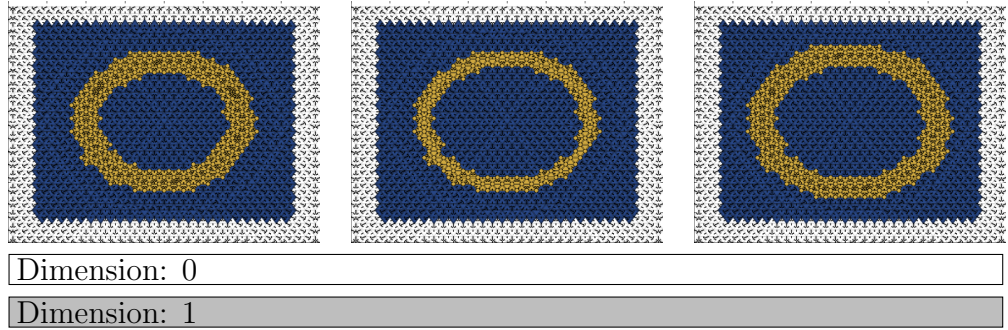
\begin{figure}[htbp]
\centering
\begin{tabular}{ccc}
  \includegraphics[height=32mm]{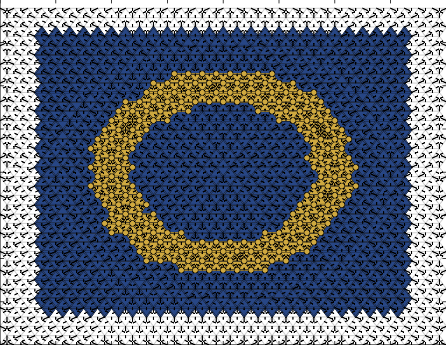}&
  \includegraphics[height=32mm]{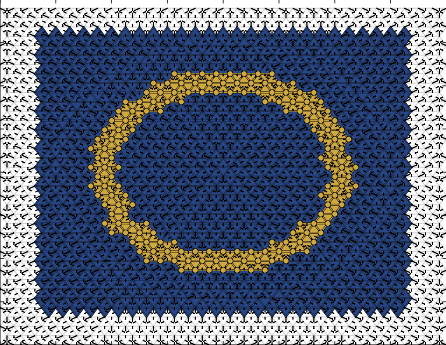}&
  \includegraphics[height=32mm]{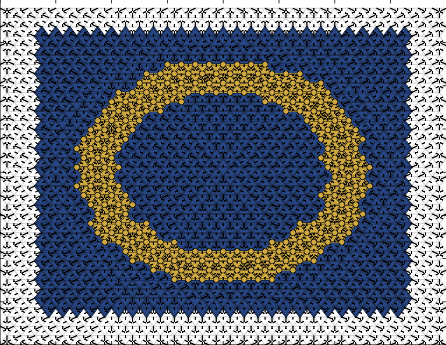}\\
  \multicolumn{3} {l} {
  \begin{tikzpicture}[outer sep = 0, inner sep = 0]
  \draw[fill=white,outer sep=0, inner sep=0] (0.0,-0.3) -- (13.21,-0.3) -- (13.21,0.1) -- (0.0,0.1) -- cycle;
  \node[align=left] at (1.3,-0.1) {Dimension: 0};
  \end{tikzpicture} }\\
  \multicolumn{3} {l} {
  \begin{tikzpicture}[outer sep = 0, inner sep = 0]
  \draw[fill=light-gray,outer sep=0, inner sep=0] (0.0,-0.3) -- (13.21,-0.3) -- (13.21,0.1) -- (0.0,0.1) -- cycle;
  \node[align=left] at (1.3,-0.1) {Dimension: 1};
  \end{tikzpicture} }
\end{tabular}
\caption{When one uses Proposition \ref{prop:thicken} to thicken the index pairs in Figure \ref{fig:unthickened}, a $1$-dimensional homology class persists through all three images, which matches the intuition.}
\label{fig:thickened}
\end{figure} 
\section{Conclusion}

We conclude by briefly discussing some directions for future work. While we have developed a method for computing a barcode of a sequence of Conley-Morse graphs, several questions remain unanswered. One area that is particularly notable is the relationship between classical and combinatorial isolated invariant sets. Throughout this paper, our implemented examples were obtained by discretizing differential equations into multivector fields. We were fortunate that the isolated invariant sets obtained by this method aligned with the isolated invariant sets from the classical dynamical systems. If one could find sufficient conditions that ensure some
kind of equivalence between the classical and combinatorial isolated invariant sets, it would open up the door to rigorous analysis of differential equations with persistent homology. For some preliminary work in this direction, utilizing triangulations with a transversality property \cite{BKM2007}, see \cite{MW19}.  

\bibliographystyle{abbrv}
\bibliography{refs}

\end{document}